\documentclass[12pt,a4paper]{amsart}

\usepackage{amsmath}
\usepackage[margin=3cm]{geometry}
\newcommand*\patchAmsMathEnvironmentForLineno[1]{%
  \expandafter\let\csname old#1\expandafter\endcsname\csname #1\endcsname
  \expandafter\let\csname oldend#1\expandafter\endcsname\csname end#1\endcsname
  \renewenvironment{#1}%
     {\linenomath\csname old#1\endcsname}%
     {\csname oldend#1\endcsname\endlinenomath}}%
\newcommand*\patchBothAmsMathEnvironmentsForLineno[1]{%
  \patchAmsMathEnvironmentForLineno{#1}%
  \patchAmsMathEnvironmentForLineno{#1*}}%
\AtBeginDocument{%
\patchBothAmsMathEnvironmentsForLineno{equation}%
\patchBothAmsMathEnvironmentsForLineno{align}%
\patchBothAmsMathEnvironmentsForLineno{flalign}%
\patchBothAmsMathEnvironmentsForLineno{alignat}%
\patchBothAmsMathEnvironmentsForLineno{gather}%
\patchBothAmsMathEnvironmentsForLineno{multline}%
}
\usepackage[utf8]{inputenc}
\usepackage[T1]{fontenc}
\usepackage[english,ngerman]{babel}
\usepackage{stmaryrd}
\usepackage{dsfont}
\usepackage[matrix, arrow, curve]{xy}
\usepackage{verbatim}
\usepackage{hyperref}
\usepackage[mathlines]{lineno}
\usepackage{enumerate}
\usepackage{tikz-cd}
\usepackage{tikz}
\usetikzlibrary{decorations.pathmorphing, decorations.markings, arrows, matrix, fit, positioning}
\tikzstyle directed=[postaction={decorate,decoration={markings, mark=at position .55 with {\arrow{stealth}}}}]	
\usepackage{color}
\usepackage[noadjust]{cite}
\usepackage{tikz}
\usetikzlibrary{matrix,arrows}
\usepackage{chemarr}
\usepackage{amssymb}

\newcommand*\C{\mathds C}

\newcommand*\Z{\mathds Z}
\newcommand*\N{\mathds N}

\newcommand*\F{\mathds F}
\newcommand*\p{\mathfrak{p}}
\renewcommand*\P{\mathfrak{P}}

\newcommand*\Q{\mathds Q}

\renewcommand*\mod{\text{ mod }}

\newcommand*\cO{\mathcal{O}}
\newcommand{\Le}{\mathcal L}

\DeclareMathOperator{\tr}{Tr}

\DeclareMathOperator{\SL}{SL}

\DeclareMathOperator{\id}{id}

\DeclareMathOperator{\Stab}{Stab}
\DeclareMathOperator{\Gal}{Gal}
\DeclareMathOperator{\Mat}{Mat}
\DeclareMathOperator{\GL}{GL}
\DeclareMathOperator{\Aut}{Aut}

\newtheorem{theorem}{Theorem}[section]
\newtheorem{lemma}[theorem]{Lemma}

\newtheorem{example}[theorem]{Example}
\newtheorem*{conjecture}{Conjecture}

\theoremstyle{definition}
\newtheorem{definition}[theorem]{Definition}
\newtheorem*{remark}{Remark}

\title{Small Heights in Large non-Abelian Extensions}
\author[L. Frey]{Linda Frey}
\address{Department of Mathematics,
Universitetsparken 5,
2100 Copenhagen
Denmark
}
\email{Linda.Frey@math.ku.dk}
\usepackage{longtable}\usepackage{fancyhdr}

\begin{document}

\pagestyle{fancy}
\fancyhf{}
\rhead{\thepage \qquad Linda Frey}
\lhead{Small Heights in Large Non-Abelian Extensions }

\selectlanguage{english}
\begin{abstract}
Let $E$ be an elliptic curve over the rationals. Let $L$ be an infinite Galois extension of the rationals with uniformly bounded local degrees at almost all primes. We will consider the infinite extension $L(E_{\text{tor}})$ of the rationals where we adjoin all coordinates of torsion points of $E$. In this paper we will prove an effective (and in the non-CM case even explicit) lower bound for the height of non-zero elements in $L(E_{\text{tor}})$ that are not a root of unity.\end{abstract}
\maketitle

\section{Introduction}
\label{Introduction}
Let $E$ be an elliptic curve defined over $\Q$, let $\mu_\infty$ be the set of all roots of unity and let $\Q(E_\text{tor})$ be the smallest field extension of $\Q$ that contains all coordinates of torsion points of $E$. In 2013 Habegger \cite{MR3090783} showed that in $\Q(E_\text{tor})^* \setminus \mu_\infty$ the height is bounded from below by a positive constant. In an earlier paper \cite{2017arXiv171204214F}, the author proves an explicit lower height bound in that case. We will generalize these results and allow larger base fields as follows.

\begin{theorem}
\label{generalization}
Let $E$ be an elliptic curve over $\Q$. Let $L$ be a (possibly infinite) Galois extension of $\Q$ with uniformly bounded local degrees by $d\in\N$. Then the height in $L(E_{\text{tor}})^* \setminus \mu_\infty$ is bounded from below by a positive constant.
\end{theorem}

We can even give an explicit formula for such a lower height bound.

\begin{theorem}
\label{generalization}
Let $E$ be an elliptic curve over $\Q$. Let $L$ be a (possibly infinite) Galois extension of $\Q$ with uniformly bounded local degrees by $d\in\N$. Let $p$ be a prime such that $p$ is surjective, supersingular and greater than $\max(2d+2,\exp(\Gal(L/\Q)))$. Then for any $\alpha \in L(E_{\text{tor}})^* \setminus \mu_\infty$ we have $h(\alpha) \geq \frac{(\log p)^4}{p^{5p^4}}$.
\end{theorem}

For the definitions of surjective and supersingular see section \ref{NotNFC}. Remark that, by  \cite{MR3009657}, we have that $\max(2d+2,\exp(\Gal(L/\Q)))$ is always finite. Furthermore, in an earlier paper \cite{2017arXiv171204214F}, the author proves an explicit upper bound for such a prime that only depends on the $j$-invariant (or the conductor, respectively) of $E$ and $\max(2d+2,\exp(\Gal(L/\Q)))$.\\

The proof of our Theorem \ref{generalization} involves the theory of local fields, ramification theory and Galois theory. In his proof, Habegger makes heavy use of the Frobenius. In our generalized case, we can not always be sure that there exists a lift of the Frobenius. We will work around that by taking suitable powers of suitable morphisms. Another key ingredient in Habegger's proof are non-split Cartan subgroups. In our proof we can completely work around that by considering the unramified and the tamely ramified case together.\\

\subsection*{Acknowledgements}
I am very thankful for all the people who helped me write this article, in particular the following. I thank Philipp Habegger who proposed this problem to me. He gave me many helpful comments and productive input. I thank Francesco Amoroso for giving me helpful comments and helping me solve the CM case. I thank Gabriel Dill for his amazing accurateness and great patience while reading my manuscripts over and over again. I thank Waltraut Lederle for helping me with my group theory issues. This research was done during my PhD\footnote{This paper is the second part of the PhD thesis of the author. The first part of the thesis is the paper \cite{2017arXiv171204214F}.} at Universit\"at Basel in the DFG project 223746744 "Heights and unlikely intersections" and written up during my SNF grant Early.PostDoc Mobility at the University of Copenhagen.\\

\section{Infinite base fields}

\begin{definition}
Let $L$ be a Galois extension of $\Q$, $S$ a set of prime numbers and $d \in\N$. We say that $L$ has \emph{uniformly bounded local degrees} above $S$ by $d$ if and only if for all primes $p \in S$ and $v$ extending $p$, we have $[L_v:\Q_p] \leq d$. Here, $L_v$ is the $v$-adic completion of $L$.
\end{definition}

Our goal is proving the following theorem.

\begin{theorem}
Let $E$ be an elliptic curve over $\Q$ and let $L/\Q$ be a Galois extension with uniformly bounded local degrees above all but finitely many primes. Then $L(E_{\text{tor}})$ has the Bogomolov property.
\end{theorem}

We will use the following result of Checcoli to make use of the uniform boundedness.

\begin{theorem}[\cite{MR3009657}]
\label{Sara}
Let $L/\Q$ be a Galois extension. Then the following conditions are equivalent:
\begin{itemize}
\item[(1)] $L$ has uniformly bounded local degrees above every prime.
\item[(2)] $L$ has uniformly bounded local degrees above all but finitely many primes.
\item[(3)] $\Gal(L/\Q)$ has finite exponent.
\end{itemize}
\end{theorem}


Remark that uniformly bounded means that the degrees are bounded independently of $p$.

In a paper of Checcoli and Zannier \cite{MR2755687} there is also the implication $(2) \Rightarrow (3)$ from the above theorem. But since we need the stronger implication $(1) \Rightarrow (3)$, we use the result of Checcoli.

\begin{example}
A field that fulfills these properties is for a fixed $d$ any subextension of $\Q^{(d)} \subset \overline{\Q}$, which is the compositum of all number fields of degree at most $d$ over $\Q$.
\end{example}

We will first prove the CM case since it follows from Theorem 1.5 of \cite{MR3182009} before we handle the more complicated non-CM case.

\begin{theorem}
Let $E$ be an elliptic curve with complex multiplication over $\Q$ and let $L/\Q$ be a Galois extension with uniformly bounded local degrees above all but finitely many primes by $d$. Then $L(E_{\text{tor}})$ has the Bogomolov property and there exists an effectively computable bound only depending on $d$.
\end{theorem}

\begin{proof}
Let $L_0$ be the CM field of $E$ and consider the following diagram

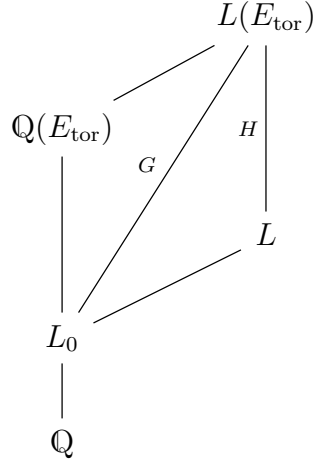
\begin{center}
\begin{tikzcd}
& L(E_\text{tor}) \\
\Q (E_\text{tor}) \arrow[dash]{ur} & \\
& L \arrow[dash]{uu}{H}\\
 L_0 \arrow[dash]{uu} \arrow[dash]{uuur}{G} \arrow[dash]{ur}  & \\
 \Q \arrow[dash]{u} &
\end{tikzcd}
\end{center}

Consider the restriction $f: G \to \Gal(\Q(E_\text{tor})/L_0) \times \Gal(L/L_0), \sigma \mapsto (\sigma|_{\Q(E_\text{tor})}, \sigma|_L)$. It is injective because $L(E_\text{tor})$ is the compositum of $L$ and $\Q(E_\text{tor})$, hence the only element that maps to the identity is the identity itself.\\

We want to show that $\Gal(L(E_\text{tor})/L) =:H$ is contained in the center of\newline $\Gal(L(E_\text{tor})/L_0) =: G$ since that will allow us to use Theorem 1.5 of \cite{MR3182009
} and immediately yield that $L(E_\text{tor})$ is Bogomolov.\\

Let $\sigma \in G$ and $\tau \in H$. Then $\sigma \tau |_{\Q(E_\text{tor})} = \tau \sigma |_{\Q(E_\text{tor})}$ since $G$ is abelian (because of $E$ having complex multiplication). Furthermore $\sigma \tau |_L = \sigma |_L$ since $\tau$ acts as the identity on $L$ and $\sigma |_L = \tau  \sigma |_L$ since the image of $\sigma$ is inside $L$ (because $L/L_0$ is Galois).\\

The bound in \cite{MR3182009} is effectively computable and only depends on the place $v$ above which we have uniformly bounded local degrees, $d$ and the degree $[L_0:\Q]$. Since we have uniformly bounded local degrees and $[L_0:\Q] = 2$, the bound only depends on $d$. \end{proof}

\begin{remark}
The same thing can be done for any CM abelian variety.
\end{remark}

\section{Local preliminaries}
\label{NotNFC}
For the rest of this paper we will fix an elliptic curve $E$ over $\Q$ without complex multiplication and with $j$-invariant $j_E$. Furthermore, fix a field $L$ with the properties from Theorem \ref{Sara} and call $d$ the uniform bound for the local degrees. We use the notation $F(N) = F(E[N])$ for a field $F$ and a natural number $N$. We need the following properties of a prime $p$.

\begin{align}
(P1) & \quad p \text{ is supersingular} \label{P1}\\
(P2) & \quad p \text{ is surjective} \label{P2}\\
(P3) & \quad p \geq \max(2d+2, \exp(\Gal(L/\Q))) \label{P3}\\
(P4) & \quad j_E \not\equiv 0, 1728 \mod p \label{P4}
\end{align}

We will fix a prime $p$ such that $p$ fulfills properties (P1), (P2), (P3) and (P4). For $N\in \N$ we let $N = p^n M$ where $M$ and $p$ are coprime. We want to consider every field as a subfield of a fixed algebraic closure $\overline{\Q_p}$ of $\Q_p$.  With $\Q_q$ we denote the unique quadratic unramified extension. The proof goes as follows: For an element $\alpha \in L(E_{\text{tor}})$ we will fix a finite Galois extension $K/\Q$ such that $K(E_{\text{tor}})$ contains $\alpha$ and $K\subset L \subset \overline{\Q_p}$. Set $q=p^2$ and call $\Q_q$ the unique quadratic unramified extension of $\Q_p$. Then we fix a Galois extension $F$ of $\Q_q$ such that: $\Q_q \subset F \subset \overline{\Q_p}$, the $v$-adic completion of $K$ is contained in $F$ (where $v$ extends $p$) and $[F:\Q_p]$ is uniformly bounded by $2d$ (since it is possible that we have to choose $F$ larger than $K_v$ so that it contains $\Q_q$). Since we consider all fields as subfields of $\overline{\Q_p}$ we can restrict the $p$-adic valuation of $\overline{\Q_p}$ to any subfield. Since all fields are Galois, the completion with respect to any place above $p$ will be the same.

For a natural number $N$, we consider $F(N)$ and deal with two cases: the wildly ramified case where $p^2 \mid N$ and the tamely ramified case where $p^2 \nmid N$. We start with a few technical lemmas.

\begin{lemma}
\label{equal}
Let $p^2 | N$. We have $\Q_q(N)\cap F = \Q_q (N/p) \cap F$.
\end{lemma}

\begin{proof}
Recall that $p > d \geq [F:\Q_p]$. By Lemma 3.4 (v) of \cite{MR3090783} we know that
\begin{align*}
\Gal(\Q_q(N)/\Q_q(N/p)) \cong (\Z/p\Z)^2
\end{align*}
in the case of $p^2 \mid N$. We consider the following diagram where the numbers next to the lines describe the degrees of the extensions:

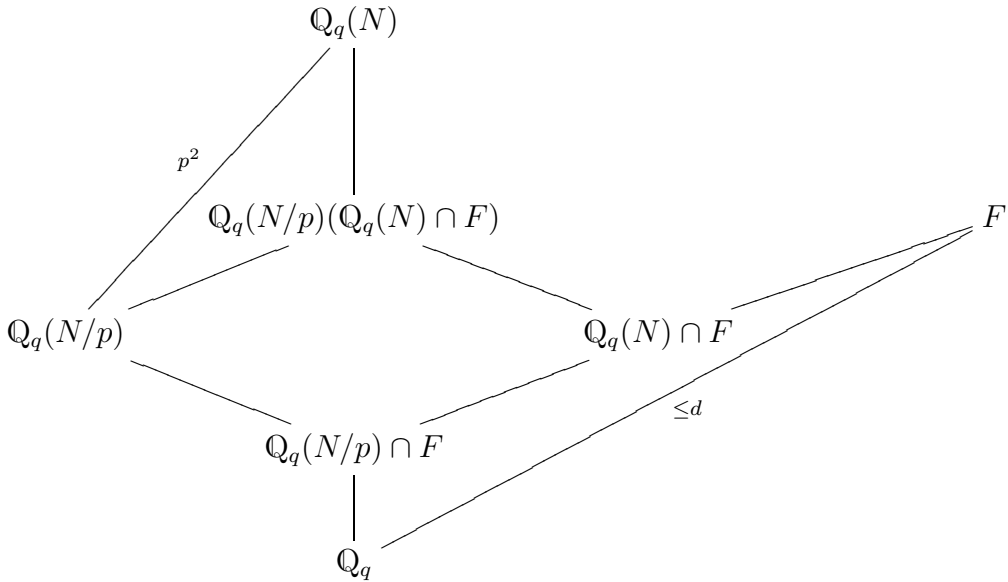
\begin{center}
\label{dasRichtigeDiagramm}
\makebox[0pt]{
\begin{xy}\xymatrix{
 & \Q_q(N) \ar@{-}_{p^2}[dddl]\ar@{-}[dd] &&&\\
 \\
 & \Q_q (N/p)(\Q_q (N) \cap F) \ar@{-}[dr]\ar@{-}[dl]& & && F\ar@{-}[dlll]\ar@{-}^{\leq d}[dddllll]\\
 \Q_q (N/p)\ar@{-}[dr] & & \Q_q(N) \cap F\ar@{-}[dl] &\\
 &\Q_q(N/p) \cap F \ar@{-}[d]&&&\\
& \Q_q & &&
}
\end{xy}
}
\end{center}

By the multiplicativity of the degree in a tower of field extensions, we have that
\begin{align*}
[\Q_q (N/p)(\Q_q (N) \cap F) : \Q_q (N/p)]  \text{ divides } p^2
\end{align*}

and by the above diagram
\begin{align*}
[\Q_q (N/p)(\Q_q (N) \cap F) : \Q_q (N/p)] = [\Q_q(N)\cap F : \Q_q (N/p) \cap F] \leq d < p.
\end{align*}

Hence $[\Q_q(N)\cap F : \Q_q (N/p) \cap F]$ must be one and the fields are equal. 
\end{proof}

\begin{lemma}\label{Lemma3.3wild}
Let $p^2 \mid N$. Then the extension $F(p^n)/F$ is abelian. Furthermore,
\begin{align*}
\Gal(F(p^n)/F(p^{n-1})) \cong (\Z/p\Z)^2.
\end{align*}
\end{lemma}

\begin{proof}
Consider the following diagram:
\begin{center}
\makebox[0pt]{
\begin{xy}\xymatrix{
 &F(p^n) = \Q_q(p^n)F(p^{n-1})\ar@{-}[dl]\ar@{-}[dr]&\\
 \Q_q(p^n) \ar@{-}[dr]&& F(p^{n-1})\ar@{-}[dl]\\
 &\Q_q(p^n)\cap F(p^{n-1})\ar@{-}[d]&\\
 &\Q_q(p^{n-1})&
}
\end{xy}
}
\end{center}
We get that $\Gal(F(p^n)/F(p^{n-1}))$ is isomorphic to a subgroup of $\Gal(\Q_q(p^n)/\Q_q(p^{n-1}))$ of index at most $[F:\Q_q]$. Since by Lemma 3.4 (v) of \cite{MR3090783} $\Gal(\Q_q(p^n)/\Q_q(p^{n-1}))$ has order $p^2$ and $[F:\Q_q]$ is strictly less than $p$, we must have $\Gal(F(p^n)/F(p^{n-1})) \cong \Gal(\Q_q(p^n)/\Q_q(p^{n-1}))$. By Lemma 3.3 (i) of \cite{MR3090783}, we get $\Gal(F(p^n)/F(p^{n-1})) \cong (\Z/p\Z)^2$. To prove that $F(p^n)/F$ is abelian, we look at the following diagram

\begin{center}
\makebox[0pt]{
\begin{xy}\xymatrix{
 &F(p^n) = \Q_q(p^n)F\ar@{-}[dl]\ar@{-}[dr]&\\
 \Q_q(p^n) \ar@{-}[dr]&& F\ar@{-}[dl]\\
 &\Q_q(p^n)\cap F\ar@{-}[d]&\\
 &\Q_q&
}
\end{xy}
}
\end{center}

So by \cite{MR3090783}, Lemma 3.4 (iv), $\Gal(F(p^n)/F)$ is isomorphic to a subgroup of $\Gal(\Q_q(p^n)/\Q_q)$ which is isomorphic to $\Z/(q-1)\Z \times (\Z/p^{n-1}\Z)^{2}$, hence both Galois groups have to be abelian.
\end{proof}

\begin{lemma}
\label{ramindex}
Let $p^2 \mid N$. The ramification index of the extension $F(p^n)/F$ is a multiple of $q^{n-1}$ and a divisor of $q^{n-1}(q-1)$. The extension $F(p^n)/F(p^{n-1})$ is totally ramified and its Galois group is isomorphic to $\Gal(\Q_q(p^n)/\Q_q(p^{n-1})) \cong (\Z/p\Z)^2$. In particular, $F(p^n)/F(p)$ is totally ramified.
\end{lemma}

\begin{proof}
We consider the following diagram:

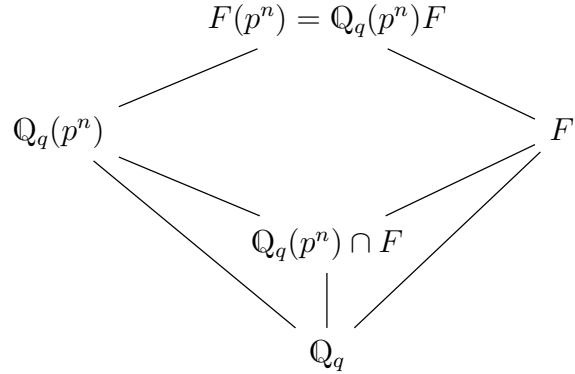
\begin{center}
\begin{tikzcd}
& F(p^n) = \Q_q(p^n)F \arrow[dash]{dl} \arrow[dash]{dr} & \\
\Q_q(p^n)\arrow[dash]{dr} \arrow[dash]{ddr}&& F \arrow[dash]{dl} \arrow[dash]{ddl}\\
& \Q_q(p^n) \cap F\arrow[dash]{d}& \\
& \Q_q &
\end{tikzcd}
\end{center}

We want to equip this diagram with the ramification indices. From Lemma 3.3 (i) of \cite{MR3090783}, we know that $\Q_q(p^n)/\Q_q$ is totally ramified of degree $(q-1)q^{n-1}$. By construction, the extension $F/\Q_q$ has degree (hence ramification index) at most $d$ which is less than $p$. Since $\Gal(F(p^n)/\Q_q(p^n))$ is isomorphic to a subgroup of $\Gal(F/\Q_q)$, its degree has to be at most $d$ hence also the ramification index. So we get the following diagram.

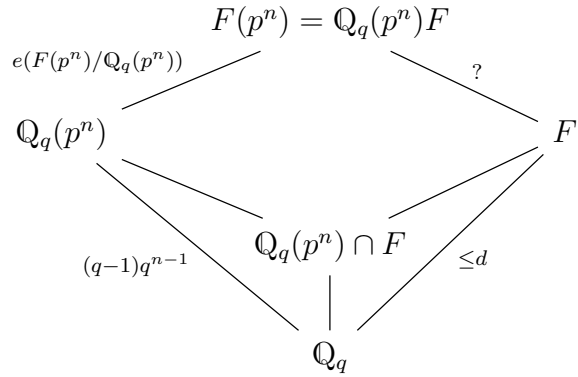
\begin{center}
\begin{tikzcd}
& F(p^n) = \Q_q(p^n)F \arrow[dash, swap]{dl}{e(F(p^n)/\Q_q(p^n))} \arrow[dash]{dr}{?} & \\
\Q_q(p^n)\arrow[dash]{dr} \arrow[dash, swap]{ddr}{(q-1)q^{n-1}}&& F \arrow[dash]{dl} \arrow[dash]{ddl}{\leq d}\\
& \Q_q(p^n) \cap F\arrow[dash]{d}& \\
& \Q_q &
\end{tikzcd}
\end{center}

This shows that the ramification index of $F(p^n)/\Q_q$ is a multiple of the ramification degree of $\Q_q(p^n)/\Q_q$ which is $(q-1)q^{n-1}$. But since the ramification degree of $F/\Q_q$ is at most $d$ which is coprime to $p$, we get that the ramification degree of $F(p^n)/F$ has to be a multiple of $q^{n-1}$.

With a similar diagram we can show that $F(p^n)/F(p^{n-1})$ is totally ramified. Recall that by Lemma \ref{equal} we have $\Q_q(p^{n-1})\cap F = \Q_q (p^n)\cap F$.  Hence also $\Q_q(p^{n-1}) = \Q_q(p^{n-1}) \cap F(p^{n-1})  = \Q_q(p^n) \cap F(p^{n-1})$.

\begin{center}
\label{diagram1}
\begin{tikzcd}
& F(p^n) = \Q_q(p^n)F(p^{n-1}) \arrow[dash]{dl} \arrow[dash]{dr} & \\
\Q_q(p^n)\arrow[dash]{dr}&& F(p^{n-1}) \arrow[dash]{dl}\\
& \Q_q(p^{n}) \cap F(p^{n-1})= \Q_q (p^{n-1}) \cap F(p^{n-1}) = \Q_q(p^{n-1})&
\end{tikzcd}
\end{center}

The ramification index of $\Q_q(p^n)/\Q_q(p^{n-1})$ is exactly $q$ and the ramification index of $F(p^n)/\Q_q(p^n)$ is at most its degree $[F(p^n):\Q_q(p^n)] \leq [F:\Q_q] < p$, hence not divisible by $p$. The same works for $F(p^{n-1})/\Q_q(p^{n-1})$. By looking at the divisibility we see that the ramification index of $F(p^n)/F(p^{n-1})$ also has to be $q$, hence it is totally ramified. Furthermore, $\Gal(F(p^n)/F(p^{n-1}))$ is isomorphic to $\Gal(\Q_q(p^n)/\Q_q(p^{n-1}))$.

\end{proof}

The following lemma is the analogue of Lemma 3.4 of \cite{MR3090783}.

\begin{lemma}
\label{Lemma3.4}
The following statements hold.
\begin{enumerate}[(i)]
\item The compositum $F(p^n)F(M)$ is $F(N)$.
\item The extension $F(N)/F(p^n)$ is unramified.
\item The Galois group $\Gal(F(N)/F(M))$ is abelian.
\item If $n \geq 2$, then $\Gal(F(N)/F(N/p)) \cong \Gal(F(p^n)/F(p^{n-1})) \cong  (\Z/p\Z)^2$ and the extension $F(N)/F(N/p)$ is totally ramified.
\item The image of the representation $\Gal(F(p^n)/F) \to \Aut(E[p^n])$ contains multiplication by $M^{[F:\Q_q]}$.
\item The Galois group $\Gal(F(p)/F)$ is isomorphic to a subgroup of $\Z/(q-1)\Z$.
\end{enumerate}
\end{lemma}

\begin{proof}
Every $N$-torsion point is the sum of a $p^n$-torsion point and an $M$-torsion point. Hence, the composition $F(p^n)F(M)$ has to be equal to $F(N)$ which is the statement in (i).\\
For (ii) we consider the following diagram:

\begin{center}
\makebox[0pt]{
\begin{xy}\xymatrix{
& F(N) = \Q_q(N)F(p^n) \ar@{-}[dl]\ar@{-}[dr]& \\
\Q_q(N)\ar@{-}[dr] && F(p^n)\ar@{-}[dl] \\
& \Q_q(N) \cap F(p^n) &\\
&\Q_q(p^n)&
}
\end{xy}
}
\end{center}
Since the extension $\Q_q(N) /\Q_q(p^n)$ is unramified by Lemma 3.4 (ii) \cite{MR3090783}, the subextension $\Q_q(N)/\Q_q(N) \cap F(p^n)$ also has to be unramified. Hence by \cite{MR1697859} Proposition 7.2, the extension $F(N)/F(p^n)$ extension also has to be unramified.\\

For (iii) we consider the following diagram:

\begin{center}
\begin{tikzcd}
& F(N) = \Q_q(p^n)F(M) \arrow[dash]{dl} \arrow[dash]{dr} & \\
\Q_q(p^n)\arrow[dash]{dr} \arrow[dash, swap]{ddr}&& F(M) \arrow[dash]{dl} \arrow[dash]{ddl}\\
& \Q_q(p^n) \cap F(M)\arrow[dash]{d}& \\
& \Q_q &
\end{tikzcd}
\end{center}

So $\Gal(F(N)/F(M))$ is isomorphic to a subgroup of $\Gal(\Q_q(p^n)/\Q_q)$ which is by \cite{MR3090783}, Lemma 3.4 (iv), isomorphic to $\Z/(q-1)\Z \times (\Z/p^{n-1} \Z)^2$, so it has to be abelian.\\

For (iv) we recall Lemma 3.4 (iv) of \cite{MR3090783}:
\begin{align*}
\Gal(\Q_q(N)/\Q_q(N/p)) \cong  \begin{cases}
(\Z/p\Z)^2 &\mbox{ if } n\geq 2,\\
\Z/(q-1)\Z &\mbox{ if } n = 1.
\end{cases}
\end{align*}

Let now $n\geq 2$. We want to use Lemma 2.1 (i) of \cite{MR3090783} with the unramified extension $F(N/p)/F(p^{n-1})$ (see Lemma \ref{Lemma3.4} (ii)) and the totally ramified extension $F(p^n)/F(p^{n-1})$ (see Lemma \ref{ramindex}). We get $F(p^{n-1}) = F(p^n)\cap F(N/p)$ and with the following diagram 
\begin{center}
\makebox[0pt]{
\begin{xy}\xymatrix{
&F(N) = F(p^n)F(N/p)\ar@{-}[dr]\ar@{-}[dl]&\\
F(p^n)\ar@{-}[dr] &&F(N/p)\ar@{-}[dl]\\
&F(p^{n-1}) = F(p^n)\cap F(N/p)&
}
\end{xy}
}
\end{center} we can use Lemma \ref{Lemma3.3wild} to get
\begin{align*}
\Gal(F(N)/F(N/p)) \cong \Gal(F(p^n)/F(p^{n-1})) \cong (\Z/p\Z)^2.
\end{align*}

With Lemma (ii) 2.1 of \cite{MR3090783}, we get that the extension $F(N)/F(N/p)$ is totally ramified. \\

We now come to part (v). By Lemma 3.3 (iii) of \cite{MR3090783}, the image of the representation $\Gal(\Q_q(p^n)/\Q_q) \to \Aut (E[p^n])$ contains multiplication by $M$. Let us call $\sigma$ the preimage of multiplication by $M$. We have a representation $\Gal(F(p^n)/\Q_q) \to \Aut (E[p^n])$ that is compatible to the above one and we can choose an element in $\Gal(F(p^n)/\Q_q)$ that restricts to $\sigma$ in $\Gal(\Q_q(p^n)/\Q_q)$. We will call this element also $\sigma$. Since $\Gal(F(p^n)/F)$ is a normal subgroup of $\Gal(F(p^n)/\Q_q)$, we can look at the projection $f: \Gal(F(p^n)/\Q_q) \to \Gal(F(p^n)/\Q_q) / \Gal(F(p^n)/F)$. The index of $\Gal(F(p^n)/\Q_q)$ in $\Gal(F(p^n)/F)$ is equal to $[F:\Q_q]$. So \begin{align*}
f(\sigma^{[F:\Q_q]}) = f(\sigma)^{[F:\Q_q]} = id.
\end{align*}
Hence $\sigma^{[F:\Q_q]} $ is an element of $\Gal(F(p^n)/F)$ and it will act as multiplication by $M^{[F:\Q_q]} $ on $E[p^n]$.\\

For (vi) we consider the following diagram.

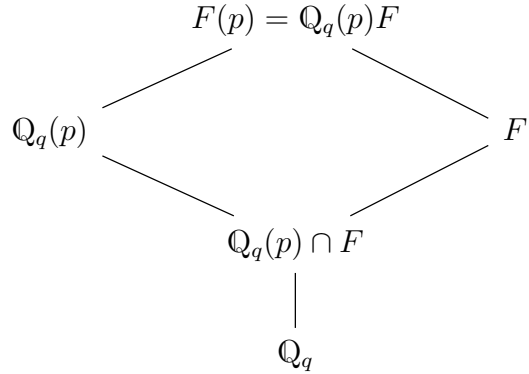
\begin{center}
\begin{tikzcd}
& F(p) = \Q_q(p)F \arrow[dash, swap]{dl} \arrow[dash]{dr} & \\
\Q_q(p)\arrow[dash]{dr} && F \arrow[dash]{dl}\\
& \Q_q(p) \cap F\arrow[dash]{d}& \\
& \Q_q &
\end{tikzcd}
\end{center}

By \cite{MR3090783}, Lemma 3.4 (iv), we know that $\Gal(\Q_q(p)/\Q_q) \cong \Z/(q-1)\Z$. So $\Gal(F(p)/F)$ has to be isomorphic to a subgroup of $\Z/(q-1)\Z$.
\end{proof}

Recall the definition of the higher ramification groups:
\begin{align*}
G_i (L/K) := \{ \sigma \in \Gal (L/K) |  \forall a \in \cO_K \text{ we have } w(\sigma (a) - a) \geq i+1 \}.
\end{align*}

\begin{lemma}
\label{HigherRamificationGroup}
Let $p^2 \mid N$. Then there is $s\geq q^{n-1}-1$ such that
\begin{align*}
\Gal(F(N)/F(N/p)) \subset G_s (F(N)/F).
\end{align*}
\end{lemma}

\begin{proof}
First, we want to show that $\Gal(F(p^n)/F(p^{n-1})) \subset G_s (F(p^n)/\Q_q)$ for some $s\geq q^{n-1}-1$.\\

By Lemma 3.3 (ii) of \cite{MR3090783}, we know that
\begin{align*}
\Gal(\Q_q(p^n)/\Q_q(p^{n-1})) = G_{q^{n-1}-1}(\Q_q(p^n)/\Q_q).
\end{align*}
So we take an element $\psi$ of \newline $\Gal(F(p^n)/F(p^{n-1}))$ and look at the restriction to $\Q_q(p^n)$ which will be an element of $\Gal(\Q_q(p^n)/\Q_q(p^{n-1}))$ and hence of $G_{q^{n-1}-1}(\Q_q(p^n)/\Q_q)$.  We will use Herbrand's Theorem (Theorem 10.7 of \cite{MR1697859}) which says that for any $s \geq -1$
\begin{align*}
(G_s (F(p^n)/\Q_q) \Gal(F(p^n)/\Q_q(p^n)))/\Gal(F(p^n)/\Q_q(p^n)) = G_t (\Q_q(p^n)/\Q_q)
\end{align*}

where $t$ depends on $s$. By Proposition IV.12 of \cite{MR554237} $t$ is given by a continuous and increasing function of $s$ that maps $0$ to $0$ and goes to infinity as $s$ goes to infinity. By the piecewiese linearity seen in the equation on p. 73 of \cite{MR554237}, we see that for $t=q^{n-1}-1$ we can find $s$ such that the above is true and $s\geq t$.\\

Now since the restriction $\psi|_{\Q_q(p^n)}$ is an element of $G_{q^{n-1}-1} (\Q_q(p^n)/\Q_q)$ we find $\sigma_1 \in G_s (F(p^n)/\Q_q)$ and $\sigma_2 \in \Gal(F(p^n)/\Q_q(p^n))$ such that $\psi = \sigma_1 \sigma_2$. Since $G_s(F(p^n)/\Q_q)$ is a normal subgroup of $\Gal(F(p^n)/\Q_q)$, we can consider \newline $\Gal(F(p^n)/\Q_q)/G_s(F(p^n)/\Q_q)$. We want to consider the homomorphism of groups
\begin{align*}
f: \Gal(F(p^n)/\Q_q) \to \Gal(F(p^n)/\Q_q)/G_s (F(p^n)/\Q_q).
\end{align*}
Since $\sigma_1 \in G_s (F(p^n)/\Q_q)$ we have $f(\sigma_1) = \id$. Furthermore,
\begin{align*}
f(\sigma_1 \sigma_2)^{[F(p^n):\Q_q(p^n)]} &= (f(\sigma_1) f(\sigma_2))^{[F(p^n):\Q_q(p^n)]}\\
&= f(\sigma_2)^{[F(p^n):\Q_q(p^n)]}\\
&= f(\sigma_2^{[F(p^n):\Q_q(p^n)]})\\
& = f(\id)\\
& = \id.
\end{align*}

So with $e := [F:\Q_q]!$ we can make sure that $(\sigma_1 \sigma_2)^e \in G_s (F(p^n)/\Q_q)$. But since $\psi$ was in $\Gal(F(p^n)/F(p^{n-1}))$ which is by Lemma \ref{Lemma3.4} (iv) isomorphic to $(\Z/p\Z)^2$ and $e$ is coprime to the order of $\Gal(F(p^n)/F(p^{n-1}))$, we can find $\tilde{\psi}\in \Gal(F(p^n)/F(p^{n-1}))$ such that $\tilde{\psi}^e = \psi$.\\

Hence we get that 
\begin{align}
\label{HigherRamification1}
\Gal(F(p^n)/F(p^{n-1})) \subset G_s (F(p^n)/\Q_q) = G
\end{align}
and we showed that there exists $s\geq q^{n-1}-1$ such that $G_s (F(p^n)/\Q_q)$ has order at least $p^2$.

Now by Lemma 2.1 (iii) of \cite{MR3090783} and with $F(N/p)/F(p^{n-1})$ unramified and $F(p^n)/F(p^{n-1})$ totally ramified, we have
\begin{align}
\label{HigherRamification2}
\Gal(F(N)/F(N/p)) \cap G_s(F(N)/F(p^{n-1})) \cong G_s (F(p^n)/F(p^{n-1}))
\end{align}
by restriction. By Lemma \ref{Lemma3.4} (iv), $\Gal(F(N)/F(N/p))$ must have order $p^2$ and since $G_s (F(p^n)/F(p^{n-1}))$ also has order $p^2$, they have to be isomorphic by restriction. By set theory, we then get
\begin{align}
\Gal(F(N)/F(N/p))& = \Gal(F(N)/F(N/p)) \cap G_s(F(N)/F(p^{n-1})) \nonumber \\
\label{HigherRamification3}&\subset G_s(F(N)/F(p^{n-1})).
\end{align}

By the formal definition of the higher ramification group we get that
\begin{align*}
G_s(F(N)/F(p^{n-1})) = G_s (F(N)/F) \cap \Gal(F(N)/F(p^{n-1})) \subset G_s(F(N)/F).
\end{align*}

Hence $\Gal(F(N)/F(N/p))$ is a subgroup of $G_s(F(N)/F)$ which is what we wanted to show.
\end{proof}

\begin{lemma}
\label{Lemma3.5}
Let $n\geq 2$. We have $F(N) \cap \mu_{p^\infty} = \mu_{p^n}$.
\end{lemma}

\begin{proof}
Since by Lemma 3.5 of \cite{MR3090783} $\Q_q(N) \cap \mu_{p^\infty} = \mu_{p^n}$, we have $F(N) \supset \mu_{p^n}$ and we only have to show "$\subset$". We will closely follow Habegger's proof of Lemma 3.5 of \cite{MR3090783} and first show that $F(p^n) \cap \mu_{p^\infty} = \mu_{p^n}$. Let $\zeta \in F(p^n)$ be a root of unity of order $p^{n^\prime}$ with $n^\prime \geq n$. By restricting we get a surjective homomorphism
\begin{align*}
\Gal (F(p^n)/F) \twoheadrightarrow \Gal(F(\zeta)/F).
\end{align*}
We will later prove that the left group is isomorphic to $(\Z/p^{n-1} \Z)^2 \times A$ where $A$ is a subgroup of $\Z/(q-1)\Z$. The right part is isomorphic to a subgroup of $\Gal(\Q_p(\zeta)/\Q_p)$ which itself is isomorphic to $\Z/p^{n^\prime-1}\Z \times \Z/(p-1)\Z$ by Proposition II.7.13 of \cite{MR1697859}. Remark that $\Z/p^{n^\prime-1}\Z \times \Z/(p-1)\Z$ is cyclic since $p-1$ and $p$ are coprime, hence all subgroups are direct products of subgroups. Since the index of $\Gal(F(\zeta)/F)$ in $\Gal(\Q_p(\zeta)/\Q_p)$ can be at most $[F:\Q_p]$ which is less than $p$, we must have that $\Gal(F(\zeta)/F)$ is actually isomorphic to $\Z/p^{n^\prime-1}\Z \times A$ where $A$ is a subgroup of $\Z/(p-1)\Z$. Recall that $\Gal F(p^n)/F)$ is isomorphic to a subgroup of $\Gal(\Q_q(p^n)/\Q_q) \cong (\Z/p^{n-1}\Z)^2 \times \Z/(p-1)\Z$. So the homomorphism can only be surjective if it maps $(\Z/p^{n-1} \Z)^2$ surjectively to $\Z/p^{n^\prime-1}\Z$ which is only possible when $n \geq n^\prime$. Together with $n^\prime \geq n$ we get that $n^\prime = n$. Let now $\zeta \in F(N)$ be a root of unity of order $p^{n^\prime}$ with $n^\prime \geq n$. The extension $F(N)/F(p^n)$ is unramified, hence also $F(p^n)(\zeta)/F(p^n)$. By the properties of the Weil pairing we know that $\zeta \in \Q_p(p^{n^\prime}) \subset F(p^{n^\prime})$. By Lemma \ref{ramindex}, the extension $F(p^{n^\prime})/F(p^n)$ is totally ramified and so is $F(p^n)(\zeta)/F(p^n)$. Hence this extension must be trivial and we have $\zeta\in F(p^n)$.\\

So let us now prove that $\Gal(F(p^n)/F) \cong (\Z/p^{n-1} \Z)^2 \times A$. Recall that $p$ and $[F:\Q_q]$ are coprime and consider the following diagrams:

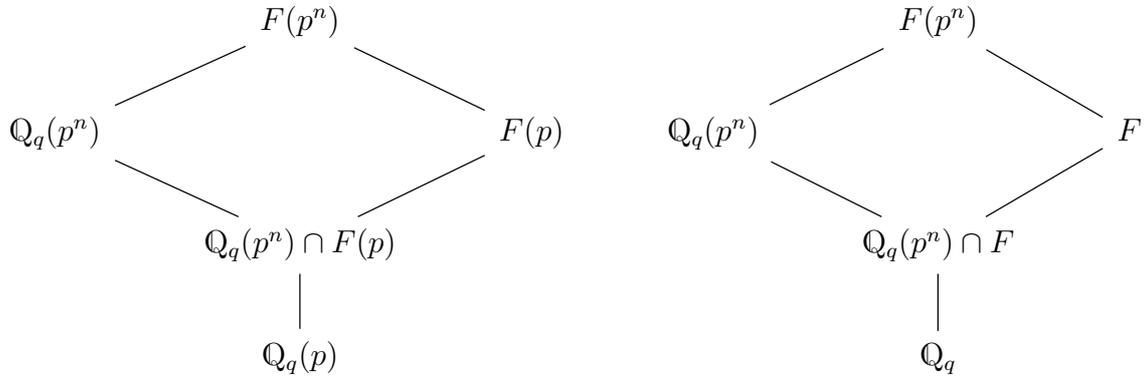
\begin{center}
\begin{tikzcd}
& F(p^n) \arrow[dash]{dl} \arrow[dash]{dr}&     &  & F(p^n)  \arrow[dash]{dl} \arrow[dash]{dr}& \\
\Q_q(p^n) \arrow[dash]{dr}&&F(p) \arrow[dash]{dl} & \Q_q(p^n) \arrow[dash]{dr}&&F \arrow[dash]{dl}\\
&\Q_q(p^n) \cap F(p) \arrow[dash]{d}&       &      &\Q_q(p^n) \cap F \arrow[dash]{d}& \\
&\Q_q (p) &       &      & \Q_q &
\end{tikzcd}	
\end{center}

The diagram on the right hand side shows that $\Gal(F(p^n)/F)$ is isomorphic to a subgroup $(\Z/p^{n-1} \Z)^2 \times \Z/(q-1)\Z$. The diagram on the left hand side shows that $\Gal(F(p^n)/F(p))$ is isomorphic to a subgoup of $(\Z/p^{n-1} \Z)^2$. By Goursat's Lemma \cite{MR1508819} and since their orders are equal, the groups have to be isomorphic.
\end{proof}

Recall that $N=p^n M$.

\begin{lemma}
Let $\psi\in\Gal(F(N)/F(N/p))$ and $\xi\in F(N) \cap \mu_M$. Then $\psi(\xi)=\xi$.	
\end{lemma}

\begin{proof}
By Proposition II 7.12 of \cite{MR1697859}, the extension $F(\xi)/F$ is unramified. 

Now we want to prove that $F(\xi) \subset F(N/p)$ (and hence $\psi(\xi)=\xi$). We know that $F(\xi)/F$ is unramified, hence $F(N/p)(\xi)/F(N/p)$ is also unramified. Furthermore, by Lemma \ref{Lemma3.4} (iv), $F(N)/F(N/p)$ is totally ramified, hence as a subextension, $F(N/p)(\xi)/F(N/p)$ also has to be totally ramified. But totally ramified and unramified extensions are trivial and we get that $F(N/p)(\xi) = F(N/p)$, hence $F(\xi)\subset F(N/p)$.
\end{proof}

\begin{lemma}
\label{Lemma3.6}
Let $N=p^n M$ with $n\geq 2$. If $\psi \in \Gal(F(N)/F(N/p))$ and $\alpha\in F(N)\setminus \{0\}$ such that $\frac{\psi(\alpha)}{\alpha}\in \mu_\infty$, then
\begin{align}
\frac{\psi(\alpha)}{\alpha} \in \mu_q.
\end{align}
\end{lemma}

\begin{proof}
We will follow the analogous proof of Lemma 3.6 in \cite{MR3090783} very closely and only change it where we need to use generalized results of this section.

We write $x^\psi$ for $\psi(x)$ if $x\in F(N)$, hence $\frac{\psi(\alpha)}{\alpha} = \alpha^{\psi-1}$. Let $N^\prime$ denote the order of $\beta := \alpha^{\psi-1}$ and decompose it as $N^\prime = p^{n^\prime} M^\prime$ with nonnegative $n^\prime$ and $M^\prime$ and $p$ coprime. Then $\xi := \beta^{p^{n^\prime}}$ has order $M^\prime$. By the lemma above, $\xi$ is fixed by $\psi$.

The order of $\beta^{M^\prime}$ is $p^{n^\prime}$. Hence $n^\prime \leq n$ by the above Lemma \ref{Lemma3.5}. For the same reason we have $\beta^{pM^\prime}\in F(N/p)$, hence $\psi$ fixes $\beta^{pM^\prime}$.

Let us write $1= ap^{n^\prime} + bM^\prime$ with $a$ and $b$ integers. Then $\beta=\xi^a\beta^{bM^\prime}$ and so $\psi$ fixes $\beta^p$ since it fixes $\xi$ and $\beta^{pM^\prime}$.

Let $t$ denote the order of $\psi$ as an element of $\Gal(F(N)/F(N/p))$. Then
\begin{align}
1 = \alpha^{p(\psi^t-1)} = \alpha^{p(\psi-1)(\psi^{t-1}+...+\psi+1)} = \beta^{p(\psi^{t-1}+...+\psi+1)} = \beta^{pt}.
\end{align}

By Lemma \ref{Lemma3.4} (iv) the order $t$ divides $p$ and the statement follows.
\end{proof}

\section{The tamely ramified case}

Again, we fix $E$, $L$, $K$ and $p$ as in section \ref{NotNFC}.\\

Remark that the tamely ramified case includes the unramified case. For the whole section let $p^2 \nmid N$ and $\varphi_q \in \Gal(\Q_q^{\text{unr}}/\Q_q)$ be the lift of the Frobenius squared as in \cite{MR3090783}. For $p\nmid N$ we let $\tilde{F}: = F$ and for $p\mid N$ we let $\tilde{F} := F(p)$. Recall that the extension $F/\Q_q$ is Galois. Recall that $N=p^n M$.

\begin{lemma}
\label{LemmaFrobTame}
Let $\mathcal{E}$ be a multiple of $[F:\Q_q](q-1)$. We have
\begin{itemize}
\item[(i)] $\varphi_q^{\mathcal{E}} |_{\tilde{F}\cap \Q_q^{\text{unr}}} = \id$.
\item[(ii)] There exists $\tilde{\varphi}$ in $\Gal(\tilde{F}(M)/\tilde{F})$ such that the restriction $\tilde{\varphi}|_{(\tilde{F} \cap \Q_q^{\text{unr}})(M)}$ coincides with the restriction $\varphi_q^\mathcal{E}|_{(\tilde{F} \cap \Q_q^{\text{unr}})(M)}$.
\item[(iii)] For $\tilde{\varphi}$ from (ii) we have that $\tilde{\varphi}|_{K(N)}$ lies in the center of $\Gal(K(N)/\Q)$.
\item[(iv)] The extension $\tilde{F} (M)/(\tilde{F} \cap \Q_q^{\text{unr}})(M)$ is totally ramified.
\item[(v)] The ramification index of $\tilde{F}(M)/\Q_q$ is at most $(q-1)[F:\Q_q] \leq \mathcal{E}$.
\end{itemize}
\end{lemma}

\begin{proof}
(i) We have that $[F:\Q_p]$ is a multiple of $|\Gal(\cO_F/\P / \cO_{\Q_q} / \p)|$ where $\P$ and $\p$ are the maximal ideals of $\cO_F$ and $\cO_{\Q_q}$, respectively. By Lemma \ref{Lemma3.4} (vi) we have that $\Gal(F(p)/F) \subset \Z/(q-1)\Z$ hence in the case of $p|N$ we have that $[\tilde{F}:\Q_q]$ is a divisor of $[F:\Q_q](q-1)$ which divides $\mathcal{E}$ and whenever $p\nmid N$, we still have that $[\tilde{F}:\Q_q] |$ divides $\mathcal{E}$. So $\mathcal{E}$ is always a multiple of the local degree $[\tilde{F}:\Q_q]$. After restriction $\varphi_q|_{\tilde{F}\cap\Q_q^{\text{unr}}}$ is an element of the Galois group $\Gal(\tilde{F}\cap \Q_q^{\text{unr}}/\Q_q)$. But the order of this group is a divisor of $\mathcal{E}$ since $\Gal(\tilde{F}\cap \Q_q^{\text{unr}} / \Q_q)$ is a quotient of $\Gal(\tilde{F}/\Q_q)$ which has order dividing $(q-1)[F:\Q_q]$. Hence, the $\mathcal{E}$-th power of $\varphi_q$ has to be the identity on $\tilde{F} \cap \Q_q^{\text{unr}}$.\\

(ii) First, we want to show that $(\tilde{F} \cap \Q_q^\text{unr}) (M) \cap \tilde{F} = \tilde{F}\cap \Q_q^\text{unr}$. The inclusion $(\tilde{F} \cap \Q_q^\text{unr}) (M) \cap \tilde{F} \supset \tilde{F}\cap \Q_q^\text{unr}$ is obvious and we have to prove "$\subset$". By Lemma 3.1 of \cite{MR3090783}, the extension $\Q_q(M)/\Q_q$ is unramified, hence $\Q_q^{\text{unr}} (M) = \Q_q^{\text{unr}}$. We have $(\tilde{F} \cap \Q_q^{\text{unr}})(M) \cap \tilde{F} \subset \Q_q^{\text{unr}} (M) \cap \tilde{F} = \Q_q^{\text{unr}} \cap \tilde{F}$.

We consider the following diagram:
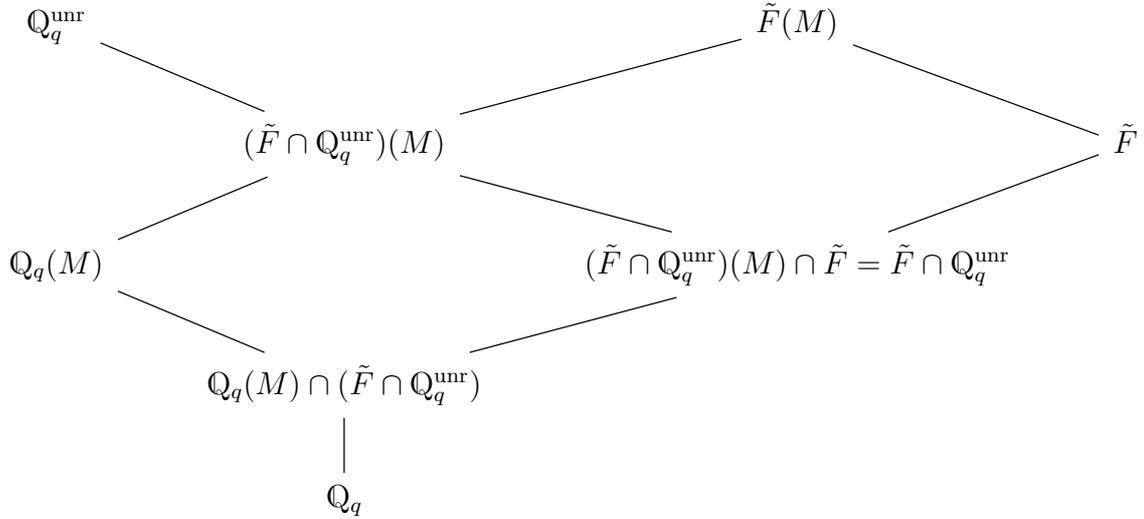
\begin{center}
\begin{tikzcd}
\Q_q^\text{unr} \arrow[dash]{dr}&&\tilde{F}(M) \arrow[dash]{dl} \arrow[dash]{dr}&&\\
 &(\tilde{F}\cap \Q_q^\text{unr})(M) \arrow[dash]{dl} \arrow[dash]{dr} & & \tilde{F} \arrow[dash]{dl}\\
 \Q_q(M) \arrow[dash]{dr} &&(\tilde{F} \cap \Q_q^{\text{unr}}) (M) \cap \tilde{F} = \tilde{F}\cap\Q_q^\text{unr} \arrow[dash]{dl}&\\
 &\Q_q(M)\cap (\tilde{F}\cap\Q_q^\text{unr}) \arrow[dash]{d}&&\\
 &\Q_q &&
\end{tikzcd}
\end{center}

Recall that $\mathcal{E}$ is a multiple of $(q-1)[F:\Q_p]$ and by (i) $\varphi_q^{\mathcal{E}} |_{\tilde{F}\cap \Q_q^\text{unr}}$ is trivial. Hence $\varphi_q^{\mathcal{E}} \in \Gal(\Q_q^\text{unr} / \tilde{F} \cap \Q_q^\text{unr})$. By the diagram, the Galois group $\Gal( (\tilde{F}\cap\Q_q^{\text{unr}})(M)/\tilde{F}\cap \Q_q^\text{unr})$ is isomorphic to $\Gal(\tilde{F}(M)/\tilde{F})$ and we call $\tilde{\varphi}$ the image of $\varphi_q |_{\tilde{F}\cap \Q_q^{\text{unr}}}$ under that isomorphism. Note that in the case of $p\mid N$, $\tilde{\varphi}$ acts trivially on $\tilde{F} = F(p) \supset K(p)$. In the case of $p\nmid N$, $\tilde{\varphi}$ acts trivially on $F \supset K$.\\

(iii) Recall that $K\subset F$ and hence by the above paragraph $\tilde{\varphi}$ acts trivially on $K$ and $K(p)$, in the cases $p\nmid N$ and $p\mid N$ respectively. We will distinguish the two cases $p\nmid N$ and $p|N$.\\

For $p|N$ we already remarked that $\tilde{\varphi}|_{K(p)}$ is the identity and we now want to show that $\tilde{\varphi} |_{K(N)}$ lies in the center of $\Gal(K(N)/\Q)$. Consider the following diagram

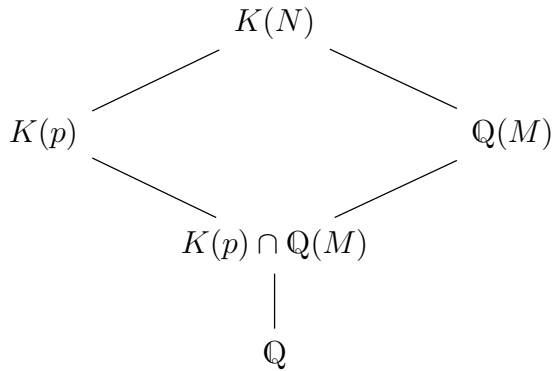
\begin{center}
\begin{tikzcd}
& K(N) \arrow[dash]{dr} \arrow[dash]{dl} & \\
K(p) \arrow[dash]{dr} && \arrow[dash]{dl} \Q(M) \\
& \arrow[dash]{d} K(p) \cap \Q (M) \\
&\Q &
\end{tikzcd}
\end{center}

which shows that $\Gal(K(N)/\Q)$ is by restriction isomorphic to a subgroup of $\Gal(\Q(M)/\Q) \times \Gal(K(p)/\Q)$. Now the proof of Lemma 5.1 of \cite{MR3090783} shows that $\tilde{\varphi}$ lies in the center of $\Gal(\Q(M)/\Q)$. Together with $\tilde{\varphi}$ acting trivially on $K(p)$, we get that it lies in the center of $\Gal(K(N)/\Q)$.\\

Now let $p\nmid N$. We do the same as above, considering $K$ instead of $K(p)$. Consider the diagram:

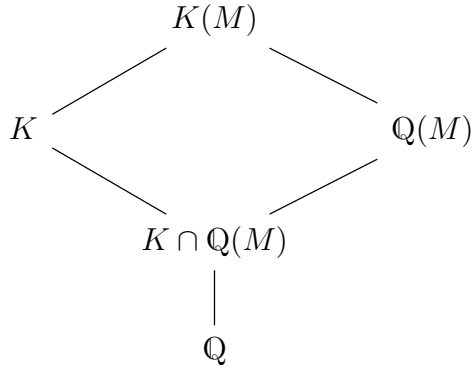
\begin{center}
\begin{tikzcd}
& K(M) \arrow[dash]{dr} \arrow[dash]{dl} & \\
K \arrow[dash]{dr} && \arrow[dash]{dl} \Q(M) \\
& \arrow[dash]{d} K \cap \Q (M) \\
&\Q &
\end{tikzcd}
\end{center}

And again: $\Gal(K(M)/\Q)$ is by restriction isomorphic to a subgroup of $\Gal(\Q(M)/\Q) \times \Gal(K/\Q)$ and since $\tilde{\varphi}$ acts trivially on $K$ and lies in the center of $\Gal(\Q(M)/\Q)$, it also lies in the center of $\Gal(K(M)/\Q)$. \\

(iv) We will use Lemma 2.1 (ii) of \cite{MR3090783} again. Since $\tilde{F}/\tilde{F}\cap \Q_p^{\text{unr}}$ is totally ramified and $(\tilde{F}\cap \Q_p^{\text{unr}})(M)/\tilde{F} \cap \Q_p^{\text{unr}}$ is unramified, the extension $\tilde{F}(M)/(\tilde{F}\cap \Q_p^{\text{unr}})(M)$ is also totally ramified.\\

(v) We consider the following diagram

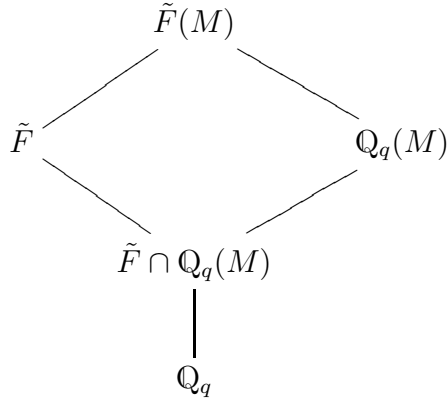
\begin{center}
\makebox[0pt]{
\begin{xy}\xymatrix{
&\tilde{F}(M) \ar@{-}[dl] \ar@{-}[dr]&\\
\tilde{F} \ar@{-}[dr]& &\Q_q (M) \ar@{-}[dl]\\
& \tilde{F} \cap \Q_q (M) \ar@{-}[d] & \\
&\Q_q &
}
\end{xy}
}
\end{center}
Since the extension $\Q_q(M)/\Q_q$ is unramified (Chapter VII \cite{MR2514094}), the only contribution to the ramification degree of $\tilde{F} / \Q_q$ can come from the extension $\tilde{F}(M)/\Q_q (M)$. Since the Galois group of the said extension is a subgroup of $\Gal (\tilde{F}/\Q_q)$, it has degree at most $(q-1)[F:\Q_q]$, hence also the ramification degree cannot be larger.
\end{proof}

Recall that since we view $F$ as a subfield of $\overline{\Q_p}$, we can consider $|\alpha|_p$ for $\alpha\in F$.

\begin{lemma}
\label{Restklassenk}
Let $L/K$ be a totally ramified extension of fields with $K \subset L \subset \overline{\Q_p}$ and $[L:\Q_p]$, $[K:\Q_p]$ finite. Then for every $\alpha\in \cO_L$ there exists $\beta\in \cO_K$ such that $|\alpha - \beta|_p < 1$.
\end{lemma}

\begin{proof}
Since the field extension is totally ramified, the residue fields are equal. Consider $\alpha$ as an element in the residue field of $L$. Take any $\beta \in \cO_K$ in the same residue class as $\alpha$. Then, as $\alpha$ and $\beta$ are in the same residue class, their difference $\alpha - \beta$ is zero in the residue field. This means $\alpha - \beta$ is an element of the maximal ideal, hence $|\alpha-\beta|_p$ has to be smaller than one.
\end{proof}

\begin{lemma}
Let $\alpha \in \tilde{F} (M)^*$ with $|\alpha|_p \leq 1$. Then for $\tilde{\varphi}$ and $\mathcal{E}$ as in Lemma \ref{LemmaFrobTame} we have
\begin{align*}
| \tilde{\varphi}(\alpha) - \alpha^{q^\mathcal{E}}|_p \leq p^{-\frac1{\mathcal{E}}}.
\end{align*}
\end{lemma}

\begin{proof}
Let $\alpha \in \tilde{F} (M)$ with $|\alpha|_p \leq 1$. Then by Lemma \ref{Restklassenk} and \ref{LemmaFrobTame} (iv) we find $\beta\in (\tilde{F}\cap\Q_q^{\text{unr}})(M)$ with $|\beta|_p \leq 1$ and $|\alpha - \beta |_p < 1$. Now $| \tilde{\varphi}(\alpha)-\tilde{\varphi}(\beta)|_p = |\alpha - \beta|_p$ since Galois automorphisms do not change the valuation. Furthermore, we have
\begin{align*}
(\alpha^{q^\mathcal{E}} - \beta^{q^\mathcal{E}}) &= (\alpha - \beta)(\alpha^{{q^\mathcal{E}}-1} + \alpha^{{q^{\mathcal{E}}}-2}\beta + ... + \alpha\beta^{{q^{\mathcal{E}}}-2} + \alpha^{{q^{\mathcal{E}}}-1})\\
\end{align*}
hence
\begin{align*}
|\alpha^{q^{\mathcal{E}}} - \beta^{q^{\mathcal{E}}}|_p &= |\alpha - \beta|_p |\alpha^{{q^\mathcal{E}}-1} + \alpha^{{q^\mathcal{E}}-2}\beta + ... + \alpha\beta^{{q^\mathcal{E}}-2} + \alpha^{{q^\mathcal{E}}-1}|_p\\
&\leq |\alpha - \beta|_p \max(|\alpha^{{q^\mathcal{E}}-1}|_p,|\alpha^{{q^\mathcal{E}}-2}\beta|_p, ..., |\alpha\beta^{{q^\mathcal{E}}-2}|_p, |\alpha^{{q^\mathcal{E}}-1}|_p)\\
&\leq |\alpha - \beta|_p \\
& < 1.
\end{align*}

Now consider $| \tilde{\varphi}(\beta) - \beta^{q^\mathcal{E}}|_p$. Since $\beta \in (\tilde{F}\cap\Q_q^{\text{unr}})(M)$, we can apply Lemma \ref{LemmaFrobTame} (ii) and get that $\tilde{\varphi}$ acts as $\tilde{\varphi}(\beta)$ is equal to $\beta^{q^\mathcal{E}}$ in the residue field. Again, as in the proof of the above lemma, this means that their difference is an element of the maximal ideal in $(\tilde{F}\cap\Q_q^{\text{unr}})(M)$, which means that $| \tilde{\varphi}(\beta) - \beta^{q^\mathcal{E}}|_p < 1$.

So we have
\begin{align*}
| \tilde{\varphi}(\alpha) - \alpha^{q^\mathcal{E}}|_p & = | \tilde{\varphi}(\alpha) - \tilde{\varphi}(\beta)+ \tilde{\varphi}(\beta)- \beta^{q^\mathcal{E}}+\beta^{q^\mathcal{E}}- \alpha^{q^\mathcal{E}}|_p \\
& \leq \max(| \tilde{\varphi}(\alpha) - \tilde{\varphi}(\beta)|_p, | \tilde{\varphi}(\beta)- \beta^{q^\mathcal{E}}|_p, |\beta^{q^\mathcal{E}} - \alpha^{q^\mathcal{E}}|_p)\\
&= \max(|\alpha-\beta|_p, | \tilde{\varphi}(\beta)- \beta^{q^\mathcal{E}}|_p)\\
&< 1.
\end{align*}
Since the valuation is discrete and we bounded the ramification degree in Lemma \ref{LemmaFrobTame} (v), it has to be at most $p^{-\frac1{\mathcal{E}}}$ which proves the statement.
\end{proof}

We recall a result of an earlier paper of the author.

\begin{lemma}[\cite{2017arXiv171204214F}]
\label{sumexpl}
Let $\delta < \frac12$ and let $\beta \in \overline{\Q^*} \setminus \mu_\infty$ be such that 
$[\Q(\beta):\Q] \geq 16$ and $h(\beta)^\frac12 \leq \frac12$. Then we have
\begin{align}
\frac{1}{[\Q(\beta):\Q]} \sum_{\tau : \Q (\beta) \hookrightarrow \C} \log |\tau (\beta) -1 | \leq \frac{40}{\delta^4} h(\beta)^{\frac12-\delta}.\label{boundthesuminexpl}
\end{align}
\end{lemma}

\begin{lemma}
\label{Lemma4.1NFC}
Let $\alpha \in \tilde{F} (M)^*$. Then for $\tilde{\varphi}$ as in Lemma \ref{LemmaFrobTame} we have
\begin{align*}
| \tilde{\varphi} (\alpha) - \alpha^{q^\mathcal{E}}|_p \leq p^{-\frac1{\mathcal{E}}} \max (1, | \tilde{\varphi}(\alpha)|_p)  \max (1, |\alpha|_p)^{q^\mathcal{E}} .
\end{align*}
\end{lemma}

\begin{proof}
For $|\alpha|_p \leq 1$ this is the above lemma. Let now $|\alpha|_p > 1$ and consider $\alpha^{-1}$. Then we can use the ultrametric triangle inequality and with the above lemma we get
\begin{align*}
|\alpha^{-q^\mathcal{E}}(\tilde{\varphi}(\alpha)-\alpha^{q^\mathcal{E}})|_p = |(\alpha^{-q^\mathcal{E}}-\tilde{\varphi}(\alpha^{-1})) \tilde{\varphi}(\alpha)|_p \leq p^{-\frac1{\mathcal{E}}}| \tilde{\varphi}(\alpha)|_p
\end{align*}
which gives the desired result.
\end{proof}

Recall that an element $\sigma \in \Gal(K(N)/\Q)$ acts on the places of $K(N)$ by $|\cdot|_{\sigma v}= |\sigma^{-1} (\cdot)|_v$.

\begin{lemma}
\label{Lemma6.3}
Let $p^2 \nmid N$. Let $\alpha\in K(N) \backslash \mu_\infty$ be non-zero. Then 
\begin{align}
\label{Equation6.10}
h(\alpha) \geq  \left(\frac{\log p}{\mathcal{E}(1+q^\mathcal{E})(1+  5\cdot2^{11})}\right)^4.
\end{align}
\end{lemma}
\begin{proof}
We follow the proof of Lemma 5.1 of \cite{MR3090783} closely. Let $x = \tilde{\varphi}|_{K(N)}(\alpha) - \alpha^{q^\mathcal{E}} \in K(N)$ where $\tilde{\varphi}|_{K(N)}$ is the lift of the Frobenius from before. This is nonzero since otherwise we would get $h(\alpha) = h(\tilde{\varphi}|_{K(N)}(\alpha)) = h(\alpha^{q^\mathcal{E}}) = {q^\mathcal{E}} h(\alpha)$ hence $h(\alpha) = 0$ which contradicts our assumption on $\alpha$. So we can use the product formula
\begin{align}
\sum_w d_w \log |x|_w = 0
\end{align}
where the sum is over all places of $K(N)$.

Let $w$ be a finite place of $K(N)$ above $p$. Then $w= \sigma^{-1} v$ for some $\sigma \in \Gal (K(N)/\Q)$ and $v$ a place above $p$ because this Galois group acts transitively on the places of $K(N)$ above $p$. By Lemma \ref{LemmaFrobTame} (iii) $\tilde{\varphi}|_{K(N)}$ and $\sigma$ commute and we get
\begin{align*}
|x|_w = |\sigma(\tilde{\varphi}|_{K(N)}(\alpha))-\sigma(\alpha)^{q^\mathcal{E}}|_v = | \tilde{\varphi}|_{K(N)}(\sigma(\alpha))-\sigma(\alpha)^{q^\mathcal{E}}|_v.
\end{align*}
Now we estimate the right-hand side from above using Lemma \ref{Lemma4.1NFC} applied to $\sigma(\alpha)$
\begin{align*}
|x|_w & = |\tilde{\varphi}| _{K(N)}(\alpha) - \alpha|_w \\
& = |\sigma(\tilde{\varphi}|_{K(N)}(\alpha)) - \sigma (\alpha) |_v\\
& = |\tilde{\varphi}|_{K(N)}(\sigma(\alpha)) - \sigma (\alpha) |_v \\
& \leq p^{-\frac1{\mathcal{E}}} \max(1,| \tilde{\varphi}|_{K(N)}(\sigma(\alpha))|_v)\max(1,|\sigma(\alpha)|_v)^{q^\mathcal{E}}\\
& = p^{-\frac1{\mathcal{E}}} \max(1,|\sigma(\tilde{\varphi}|_{K(N)}(\alpha))|_v)\max(1,|\sigma(\alpha)|_v)^{q^\mathcal{E}}\\
& = p^{-\frac1{\mathcal{E}}} \max(1,| \tilde{\varphi}|_{K(N)}(\alpha)|_w)\max(1,|\alpha|_w)^{q^\mathcal{E}}.
\end{align*}
For an arbitrary finite place $w$ of $K(N)$, the ultrametric triangle inequality gives
\begin{align*}
|x|_w \leq \max(| \tilde{\varphi}|_{K(N)}(\alpha)|_w,|\alpha^{q^\mathcal{E}}|_w) \leq \max(1,| \tilde{\varphi}|_{K(N)}(\alpha)|_w)\max(1,|\alpha|_w)^{q^\mathcal{E}}.
\end{align*}
For the infinite places $w$ we have to take a little detour. We define
\begin{align*}
\beta = \frac{\tilde{\varphi}|_{K(N)}(\alpha)}{\alpha^{q^\mathcal{E}}} \in \overline{\Q} \setminus \{1\}
\end{align*}
and bound
\begin{align*}
|x|_w &= |\beta -1|_w |\alpha^{q^\mathcal{E}}|_w \leq |\beta -1|_w   \max(1,|\alpha^{q^\mathcal{E}}|_w) \\
&\leq |\beta -1|_w  \max(1,| \tilde{\varphi}|_{K(N)}(\alpha)|_w) \max(1,|\alpha^{q^\mathcal{E}}|_w)
\end{align*}
instead. We get
\begin{align*}
0  = &\sum_w d_w \log |x|_w \\
 = &\sum_{w|p} d_w \log |x|_w + \sum_{w\nmid p, w \nmid \infty} d_w \log |x|_w + \sum_{w|\infty} d_w \log |x|_w\\
 \leq& \sum_{w|p} d_w \log (p^{-\frac1{\mathcal{E}}} \max(1,| \tilde{\varphi}|_{K(N)}(\alpha)|_w)\max(1,|\alpha^{q^\mathcal{E}}|_w))\\
 &+ \sum_{w\nmid p, w \nmid \infty} d_w \log (\max(1,| \tilde{\varphi}|_{K(N)}(\alpha)|_w)\max(1,|\alpha^{q^\mathcal{E}}|_w)) \\
 &+ \sum_{w|\infty} d_w \log (|\beta -1|_w  \max(1,| \tilde{\varphi}|_{K(N)}(\alpha)|_w) \max(1,|\alpha^{q^\mathcal{E}}|_w)).
\end{align*}
After dividing by $[K(N):\Q]$ this gives
\begin{align}
\frac{\log p}{\mathcal{E}} - \frac1{[K(N):\Q]} \sum_{w|\infty} d_w \log |\beta -1|_w \leq (1+q^\mathcal{E})h(\alpha).\label{boundNFC1}
\end{align}
Let us now assume that $h(\beta) \leq \frac14, [ \Q (\beta):\Q ] \geq 16$ and $h(\alpha) \leq 1$. This is without loss of generality since otherwise the conclusion of the Lemma is clear.
By Lemma \ref{sumexpl} with $\delta = \frac14$ we get
\begin{align*}
\frac1{[K(N):\Q]} \sum_{\tau: \Q (\beta) \to \C} \log |\tau(\beta) -1| & = \frac1{[\Q(\beta):\Q]} \sum_{\tau: \Q (\beta) \to \C} \log |\tau(\beta) -1|\\
& \leq 5\cdot2^{11} h(\beta)^{\frac14}\\
& \leq 5\cdot2^{11} ((1+q^\mathcal{E})h(\alpha))^{\frac14}.
\end{align*}
Together with estimate \eqref{boundNFC1} we get
\begin{align*}
\frac{\log p}{\mathcal{E}} - 5\cdot2^{11}((1+q^\mathcal{E})h(\alpha))^{\frac14} &\leq (1+q^\mathcal{E})h(\alpha).
\end{align*}
Hence
\begin{align*}
\frac{\log p}{\mathcal{E}} &\leq (1+q^\mathcal{E})(1+ 5\cdot2^{11}) h(\alpha)^{\frac14}
\end{align*}
which gives
\begin{align*}
\left(\frac{\log p}{\mathcal{E}(1+q^\mathcal{E})(1+ 5\cdot2^{11})}\right)^4 &\leq h(\alpha).
\end{align*}
\end{proof}

\newpage

\section{The wildly ramified case}

Again, we fix $E$, $F$, $K$, $p$ and $p^2 = q$ as in section \ref{NotNFC}. For this whole section we will only consider the case where $p^2|N$. Let $v$ be the place of $F$ above $p$. Recall that we considered $F$ as a subfield of $\overline{\Q_p}$, hence for an element $\alpha\in F$ we can consider $|\alpha|_p$.

\begin{lemma}
Let $\psi \in \Gal(\Q_q(N)/\Q_q(N/p))$. Then $\psi|_{F\cap \Q_q (N)} = \id$.
\end{lemma}

\begin{proof}
By Lemma \ref{equal} we have $\Q_q(N)\cap F = \Q_q (N/p) \cap F$. Since $\psi \in \Gal(\Q_q(N)/\Q_q(N/p))$, $\psi$ must be the identity on $\Q_q(N/p)$, hence also on $\Q_q(N/p) \cap F = \Q_q(N) \cap F$.
\end{proof}

\begin{lemma}
\label{Lemma4.2NFC}
Let $\alpha \in F(N)$. Then 
\begin{align}
|\psi(\alpha)^q-\alpha^q|_p \leq p^{-1} \max (1, |\psi(\alpha)|_p)^q \max(1,|\alpha|_p)^q
\end{align}
for all $\psi \in \Gal(F(N)/F(N/p))$.
\end{lemma}

\begin{proof}
First, we suppose $|\alpha|_p \leq 1$. Let $\psi\in \Gal(F(N)/F(N/p)) $ and consider the restriction $\psi|_{F(p^n)} \in \Gal(F(p^n)/F(p^{n-1}))$. By Lemma \ref{HigherRamificationGroup} this is an element of $G_i(F(N)/F)$ for $i=q^{n-1}-1$.

By the definition of the ramification group, this means
\begin{align*}
\psi(\alpha) -\alpha \in \P^{q^{n-1}}
\end{align*}
where $\P$ is the maximal ideal in the ring of integers of $F(N)$. By Lemmas \ref{ramindex} and \ref{Lemma3.4} (ii) the ramification index $e$ of $F(N)/F$ is at most $q^{n-1}(q-1) \leq q^n$. Therefore, $(\psi(\alpha)-\alpha)^q \in \P^{q^n} \subset \P^e$. Since $p \equiv 0 \mod \P^e$ we conclude
\begin{align*}
(\psi(\alpha)-\alpha)^q \equiv \psi(\alpha)^q - \alpha^q \mod \P^e.
\end{align*}
This leads to $|\psi(\alpha^q)-\alpha^q|_p \leq |p|_p = p^{-1}$. Hence the statement follows if $|\alpha|_p \leq 1$.
Now for $|\alpha|_p > 1$ consider $\alpha^{-1}$ with $|\alpha^{-1}|_p \leq 1$. We get $|\psi(\alpha)^{-q}-\alpha^{-q}|_p \leq p^{-1}$ and
\begin{align*}
|\alpha^{-q}(\psi(\alpha)^q-\alpha^q)|_p = |(\alpha^{-q}-\psi(\alpha)^{-q})\psi(\alpha)^q|_p \leq p^{-1}|\psi(\alpha)^q|_p.
\end{align*}
After multiplying by $|\alpha^q|_p$ we obtain our statement.
\end{proof}

\begin{lemma}
\label{Lemma5.2}
Let $\psi\in \Gal(K(N) / K(N/p))$ and $v$ be the place of $K(N)$ above $p$. Let $$G = \{ \sigma \in \Gal (K(N) / \Q) | \sigma \psi \sigma^{-1} = \psi\}$$ be the centralizer of $\psi$. Then $$|Gv| \geq\frac{[K(N):\Q]}{p^4 d_v}.$$
\end{lemma}

\begin{proof}
Let $H := \Gal (K (N) /K (N/p))$, it is a normal subgroup of $\Gal (K(N) / \Q)$. The orbit of $\psi$ under conjugation by $\Gal(K(N)/\Q)$ is contained in $H$. The stabilizer of this action is $G$ so we can use the orbit-stabilizer theorem. Furthermore, by the proof of Lemma 5.2 of \cite{MR3090783}, we have $|\Gal(\Q(N)/\Q(N/p))| \leq p^4$. Since $H$ is isomorphic to a subgroup of that group, we have $|H| \leq p^4$. We get

\begin{align*}
|G| \geq \frac{|\Gal(K(N)/\Q)|}{|H|} = \frac{[K(N):\Q]}{|H|} \geq \frac{[K(N):\Q]}{p^4}.
\end{align*}

Furthermore, again by the orbit-stabilizer theorem, for a place $v$ of $K(N)$ above $p$ we have 
\begin{align}
\label{4.21}
|Gv| = \frac{|G|}{|\Stab_G (v)|} \geq \frac{[K(N):\Q]}{p^4 |\Stab_G (v)|}.
\end{align}
The Galois group $\Gal(K(N)/\Q)$ acts transitively on all places of $K(N)$ lying above $p$ and the total number of such places is $\frac{[K(N):\Q]}{d_v}$ since $K(N)$ is a Galois extension of $\Q$.
The number of places is by the orbit-stabilizer theorem again the same as $$\frac{|\Gal(K(N)/\Q)|}{|\Stab_{\Gal(K(N)/\Q)} (v)|}.$$

This gives us the following inequality: $$|\Stab_G (v)| \leq |\Stab_{\Gal(K(N)/\Q)} (v)| = d_v.$$

After inserting this in equation \eqref{4.21} we get the desired result
\begin{align}
|Gv| &\geq \frac{[K(N):\Q]}{p^4 |\Stab_G (v)|} \geq \frac{[K(N):\Q]}{p^4 d_v}.
\end{align}
\end{proof}

The next height bound is the analogue of Lemma 5.3 of \cite{MR3090783}.

\begin{lemma}
\label{Lemma5.3}
Let $\alpha\in K(N) \backslash \mu_\infty$ be non-zero and let $n\geq 2$ be the greatest integer with $p^n \mid N$. If $\alpha^q\notin F(N/p)$, then
\begin{align}\label{Estimate5.8}
h(\alpha) \geq \frac{(\log p)^4}{4\cdot 10^6 p^{32}}.
\end{align}
\end{lemma}

\begin{proof}By hypothesis we may chose $\psi \in \Gal(F(N)/F(N/p))$ with $\psi (\alpha^q) \neq \alpha^q$.

We let $$x = \psi(\alpha^q) - \alpha^q$$ and observe $x\neq 0$ by our choice of $\psi$. So
\begin{align}
\sum_v d_v \log |x|_v = 0 \label{Summe5.9}
\end{align}
by the product formula.

Say $G= \{ \sigma \in \Gal (K(N) / \Q) | \sigma \psi \sigma^{-1} = \psi\}$ as in Lemma \ref{Lemma5.2} and $v$ is the place of $K(N)$ coming from the restriction of the $p$-adic valuation on $\overline{\Q_p}$ to $K(N)$. For $\sigma\in G$ we have $$|(\sigma\psi\sigma^{-1}) (\alpha^q) - \alpha^q|_{\sigma v} = |\psi(\sigma^{-1}(\alpha^q))- \sigma^{-1}(\alpha^q)|_p.$$

We may apply Lemma \ref{Lemma4.2NFC} to $\sigma^{-1}(\alpha)$. This implies

\begin{align*}
|(\sigma\psi\sigma^{-1})(\alpha^q)-\alpha^q|_{\sigma v} & = |\psi(\sigma^{-1}(\alpha))^q- \sigma^{-1}(\alpha)^q|_v\\
& \leq p^{-1}\max \{1,|\psi(\sigma^{-1}(\alpha))|_v\}^q\max\{1,|\sigma^{-1}(\alpha)|_v\}^q\\
&\leq p^{-1}\max \{1,|(\sigma\psi\sigma^{-1})(\alpha)|_{\sigma v}\}^q\max\{1,|\alpha|_{\sigma v}\}^q.
\end{align*}
Now $\sigma\psi\sigma^{-1} = \psi$ since $\sigma\in G$. Therefore,
\begin{align}
|x|_w  \leq p^{-1}\max\{1,|\psi(\alpha)|_w\}^q\max\{1,|\alpha|_w\}^q \text{ for all } w\in Gv. \label{Estimate5.10}
\end{align}
If $w$ is an arbitrary finite place of $K(N)$, the ultrametric triangle inequality implies
\begin{align}
|x|_w \leq \max\{1,|\psi(\alpha)|_w\}^q\max\{1,|\alpha|_w\}^q. \label{Estimate5.11}
\end{align}
Now let $w$ be an infinite place. We define $$\beta = \frac{\psi(\alpha^q)}{\alpha^q} \in\overline{\Q} \backslash \{1\}$$
and bound the following expression instead:
\begin{align}
|x|_w = |\beta - 1|_w |\alpha^q|_w \leq |\beta - 1|_w \max\{1,|\alpha|_w\}^q.
 \label{Estimate5.12}
\end{align}
We split the sum \eqref{Summe5.9} up into the finite places in $Gv$, the remaining finite places, and the infinite places and the continue like in the proof of Lemma\ref{Lemma6.3}. The estimates \eqref{Estimate5.10}, \eqref{Estimate5.11} and \eqref{Estimate5.12} together with the product formula \eqref{Summe5.9} imply
\begin{align}
0 \leq &\sum_{w\in Gv} d_w (\log p^{-1} +q\log(\max\{1,|\psi(\alpha)|_w\}\max\{1,|\alpha|_w\})) \nonumber\\
&+ \sum_{w\nmid\infty, w\notin Gv} d_w q \log(\max\{1,|\psi(\alpha)|_w\}\max\{1,|\alpha|_w\}) \nonumber\\
&+ \nonumber \sum_{w\mid\infty} d_w (\log|\beta-1|_w + q\log\max\{1,|\alpha|_w\})\\
= &\sum_{w\in Gv} d_w \log p^{-1} \nonumber\\
&+ \sum_{w\nmid\infty} d_w q \log(\max\{1,|\psi(\alpha)|_w\}\max\{1,|\alpha|_w\}) \nonumber\\
\nonumber&+ \sum_{w\mid\infty} d_w (\log|\beta-1|_w + q\log\max\{1,|\alpha|_w\})\\
\leq &\sum_{w\in Gv} d_w \log p^{-1} \label{Estimate5.13}\\
&+ \sum_{w} d_w q \log(\max\{1,|\psi(\alpha)|_w\}\max\{1,|\alpha|_w\}) \nonumber\\
\nonumber&+ \sum_{w\mid\infty} d_w \log|\beta-1|_w.
\end{align}
Moreover, since the action of the Galois group is transitive and all fields here are Galois over $\Q$, all local degrees $d_w$ with $w \in Gv$ equal $d_v$. So 
\begin{align*}
\sum_{w\in Gv} d_w \log p^{-1} = d_v \log p^{-1} |Gv| \leq -d_v \log p \frac{ [K(N):\Q]}{p^4 d_v} = - \frac{\log p}{p^4} [K(N):\Q]
\end{align*}
by Lemma \ref{Lemma5.2}. We use this estimate together with \eqref{Estimate5.13} and after dividing by $[K(N):\Q]$ we obtain 
\begin{align*}
0 \leq -\frac{\log p}{ p^4} + \frac1{[K(N):\Q]}\left( \sum_{w\mid\infty} d_w \log |\beta - 1|_w \right) + qh(\psi(\alpha)) + qh(\alpha).
\end{align*}
Also, $h(\psi(\alpha)) = h(\alpha)$ and $q = p^2$, hence
\begin{align*}
\frac{\log p}{p^4} \leq \frac1{[\Q(\beta):\Q]}\left( \sum_{\tau: \Q(\beta) \hookrightarrow \C} \log |\tau(\beta) - 1| \right) + 2p^2h(\alpha).
\end{align*}

By construction we certainly have $\beta \neq 0,1$ and in order to apply Lemma \ref{sumexpl} it remains to show that $\beta$ is not a root of unity. If we assume the contrary, then $\frac{\psi(\alpha)}{\alpha}$ will be a root of unity too. Lemma \ref{Lemma3.6} implies $\left(\frac{\psi(\alpha)}{\alpha}\right)^q = 1$ which contradicts our assumption on $\alpha$.

We have $h(\beta) \leq h(\psi(\alpha^q)) + h(\alpha^q) \leq 2p^2 h(\alpha)$. Assuming $h(\beta) \leq \frac14$ and $h(\alpha) \ leq 1$ (which we can do since otherwise we would have a lower bound for $h(\alpha)$ that is better than the claim), we apply Lemma \ref{sumexpl} with $\delta = \frac14$ and get:
\begin{align*}\frac{1}{[\Q(\beta):\Q]} \sum_{\tau : \Q (\beta) \hookrightarrow \C} \log |\tau (\beta) -1 | &\leq 5\cdot2^{11} h(\beta)^{\frac14}\\
&\leq  5\cdot2^{11} (2p^2 h(\alpha))^{\frac14}
\end{align*}
and hence
\begin{align*}
\frac{\log p}{p^4} &\leq  5\cdot2^{11} (2p^2 h(\alpha))^\frac14 + 2p^2 h(\alpha)\\
& \leq 5\cdot2^{11} (2p^2 h(\alpha))^\frac14  + 2p^2 h(\alpha)^\frac14\\
& \leq (5\cdot2^{11} (2p^2)^\frac14+2p^2) h(\alpha)^\frac14.
\end{align*}

We solve the above inequality for $h(\alpha)$:
\begin{align*}
h(\alpha) & > \left(\frac{\log p}{(5\cdot2^{11} (2p^2)^{\frac14}+2p^2)p^4}\right)^4\\
& \geq \left(\frac{\log p}{(5\cdot2^{11} 2^\frac14 p^\frac12+2p^4)p^4}\right)^4\\
& \geq \left(\frac{\log p}{(6090 p^{-\frac72}+1)2p^8}\right)^4\\
& \geq \left(\frac{\log p}{46 p^8}\right)^4\\
&= \frac{(\log p)^4}{4\cdot 10^6 p^{32}}.
\end{align*}
\end{proof}

\section{Descent and the final bound}

Again, we fix $E$, $L$ and $p$ as in section \ref{NotNFC}. Let also $\mathcal{E}$ be a multiple of $[F:\Q_p](q-1)$. Now we want to turn the conditional bound in the ramified case in an unconditional bound using some descent method. First, we construct a useful automorphism of $K(N)/\Q$.

\begin{lemma}
\label{Lemma6.4}
Let $n\geq 0$ be the greatest integer with $p^n\mid N$. There exists $\sigma_F\in\Gal(F (N)/F)$, lying in the center of $\Gal(K(N)/K)$ such that $\sigma_F(\zeta)=\zeta^{4^{[F:\Q_q]}}$ for all $\zeta\in\mu_{p^n}$. Moreover, $\sigma_F$ acts on $E[p^n]$ as multiplication by $2^{[F:\Q_q]}$.
\end{lemma}

\begin{proof}
Before we prove this Lemma, let us recall that by Lemma \ref{Lemma3.5} $F(N)$ contains $\mu_{p^n}$.\\

Since $p$ is odd, Lemma \ref{Lemma3.4} (v) implies that there is $\sigma_F^\prime \in \Gal(F (p^n)/F)$ that acts on $E[p^n]$ as multiplication by $2^{[F:\Q_q]} $.
Since the Weil pairing $\langle\cdot,\cdot\rangle$ is surjective, we can find for every root of unity $\zeta\in \mu_{p^n}$ points $P,Q \in E[p^n]$ such that $\langle P,Q\rangle = \zeta$. Now $\sigma^\prime_F (\langle P,Q\rangle ) = \langle \sigma^\prime_F (P), \sigma^\prime_F (Q) \rangle = \langle 2^{[F:\Q_q]} P, 2^{[F:\Q_q]} Q \rangle = \langle P,Q \rangle^{4^{[F:\Q_q]}}$. Hence $\sigma^\prime_F$ acts on $\mu_{p^n}$ as raising to the $4^{[F:\Q_q]}$-th power.\\

We will now lift the automorphism $\sigma_F^{\prime q-1}$ to $\sigma_F^{q-1}\in\Gal(F (N)/F (M))$. For that we consider the following diagram

\begin{center}
\begin{tikzcd}
& \arrow[dash]{dl} \arrow[dash]{dr}F(N) & \\
F(p^n) \arrow[dash]{dr}&& F(M)\arrow[dash]{dl} \\
& F(p^n)\cap F(M) \arrow[dash]{d} &\\
&F&
\end{tikzcd}
\end{center}
and we will prove that $\sigma_F^{\prime q-1}|_{F(p^n)\cap F(M)}$ is the identity.

We know that $F(M)/F$ is unramified by Lemma \ref{Lemma3.4} (ii), hence its subextension $F(p^n) \cap F(M) /F$ is also unramified. But on the other hand, $F(p^n)/F(p)$ is totally ramified by Lemma \ref{ramindex}, hence $F(p^n) \cap F(M)$ has to be a subfield of $F(p)$ which has degree $q-1$ over $F$. Hence we get that $[F(p^n) \cap F(M) : F]$ divides $q-1$.

So $\sigma_F^{\prime q-1}$ is already in $\Gal(F(p^n)/F(p^n)\cap F(M))$ which is by the above diagram isomorphic to $\Gal(F(N)/F(M))$. We will call the image of $\sigma_F^{\prime q-1}$ under this isomorphism $\sigma_F^{q-1}$.\\
Taking the sum of points gives an isomorphism between $E[p^n] \times E[M]$ and $E[N]$ which is compatible with the action of $\Gal(K(N)/K)$. Since $\sigma_F $ acts as multiplication by $2^{[F:\Q_q]}$ on $E[p^n]$ and trivially on $E[M]$ and $F$ (hence also on $K$), it must lie in the center of $\Gal(K(N)/K)$.
\end{proof}

\section{Some group theory}

\begin{lemma}
For $p\neq 2$ the vector space $V := \{ A\in \Mat_2 (\F_p) | \tr A = 0 \}$ has only trivial linear subspaces that are invariant under conjugation with $\SL_2 (\F_p)$.	
\end{lemma}

\begin{proof}
Let $U$ be a non-trivial linear subspace of $V$ that is invariant under conjugation by $\SL_2 (\F_p)$. By considering the non-degenerate scalar product $\langle A,B \rangle := \tr (A^T B)$ on $V$ we can show that $U^\perp$ is also invariant: Let $A \in U^\perp$ and $B\in U$. Then for any $S \in \SL_2 (\F_p)$ we have $\tr((SAS^{-1})^T B) = \tr(S^{-T} (A^T S^T B)) = \tr((A^T S^T B) S^{-T} ) = \tr(A^T (S^T B S^{-T})) = \tr (A^T B^\prime)$ for some $B^\prime \in U$ since $U$ is invariant under conjugation. But then $\tr(A^T B^\prime) = 0$, hence $SAS^{-1} \in U^\perp$. \\
We know that $V$ has dimension three. Now if $U$ is an invariant linear subspace of dimension $2$, its orthogonal complement has to be of dimension one and we get that $V$ has only trivial invariant subvector spaces if and only if it does not have a one-dimensional invariant subvector space which is what we will prove now.\\

Let $U \subset V$ be invariant under conjugation by $\SL_2 (\F_p)$ and of dimension one. Then there must be a matrix $A = \begin{pmatrix} a & b \\ c & -a \end{pmatrix}$ non-zero, such that $$U = \{ 0, A, 2A, \ldots, (p-1)A \}.$$ Consider $S = \begin{pmatrix} 1 & 1 \\0 & 1 \end{pmatrix}$ and $$S \begin{pmatrix} a & b \\c & -a \end{pmatrix} S^{-1} = \begin{pmatrix} a+c & -2a-c+b \\c & -a-c \end{pmatrix} \stackrel{!}{=} \lambda \begin{pmatrix} a & b \\c & -a \end{pmatrix}.$$
So in order for $U$ to be invariant, we have to have $c=0$ and $\lambda = 1$, which also gives $a=0$. Hence $U$ must be the matrices of the form $\begin{pmatrix} 0 & b \\0 & 0 \end{pmatrix}$. Let us assume that space of matrices is invariant. Then the orthogonal complement is $$U^\perp = \left\{ \begin{pmatrix} x & 0 \\y & z \end{pmatrix} \in \Mat_2 (\F_p)\right\}.$$ But here we can again find that conjugation by $S$ does not stay within $U^\perp$. Let $A = \begin{pmatrix} 0 & 0 \\1 & 0 \end{pmatrix} \in U^\perp$. Then

$$SAS^{-1} = \begin{pmatrix} 1 & -1 \\1 & -1 \end{pmatrix} \notin U^\perp.$$

We excluded all possibilities of one-dimensional invariant linear subspaces and proved the lemma.
\end{proof}

\begin{lemma}
\label{containsSL2}
Let $p$ be as in section \ref{NotNFC}. Then $\rho(\Gal(K(p)/K))$ contains $ \SL_2 (\F_p)$.
\end{lemma}

\begin{proof}
By property (P2), we have $\rho(\Gal(K(p)/\Q)) = \GL_2(\F_p)$. Consider the normal subgroup $N:= \Gal(K(p)/K)$ of $\Gal(K(p)/\Q)$. Then $\Gal(K(p)/\Q)/N$,  which is isomorphic to $\Gal(K/\Q)$, has exponent $\exp(\Gal(K/\Q))$. So also $\rho(\Gal(K(p)/\Q))/\rho(N)$ has exponent dividing $\exp(\Gal(K/\Q))$. Consider $\begin{pmatrix}
	1 & 1 \\
	0 & 1
\end{pmatrix} \in \SL_2(\F_p)$. Hence $\begin{pmatrix}
	1 & 1 \\
	0 & 1
\end{pmatrix} \in \rho(\Gal(K(p)/\Q))$. We take the $\exp(\Gal(K/\Q)) $-th power of this matrix and get an element of $\rho(N)$ (recall that $\exp(\Gal(K/\Q))$ is coprime to $p$): $$\begin{pmatrix}
1 & 1 \\
0 & 1	
\end{pmatrix}^{\exp(\Gal(K/\Q))} =
\begin{pmatrix}
1 & \exp(\Gal(K/\Q)) \\
0 & 1	
\end{pmatrix} \in \rho(N) \cap \SL_2(\F_p).
$$

Since $\rho(N)$ is normal in $\rho(\Gal(K(p)/\Q))$, then also $\rho(N) \cap \SL_2(\F_p)$ is normal in $\SL_2(\F_p)$. By Theorem 8.4 of \cite{MR1878556}, the only normal subgroups of $\SL_2(\F_p)$ are $\{1\}, \{\pm 1 \}, \SL_2(\F_p)$. But since we found one element in $\rho (N) \cap \SL_2 (\F_p)$ that is not in $\{\pm 1 \}$, we get $\rho(N) \cap \SL_2 (\F_p) = \SL_2(\F_p)$. 
\end{proof}

\begin{lemma}
\label{isMat2}
Let $p\geq 3$ and $G$ be a subgroup of $\Mat_2(\F_p)$ of order $p^2$. Let $V$ be the subgroup of $\Mat_2(\F_p)$ generated by $A B A^{-1}$ where $B$ varies over $G$ and $A$ varies over $\SL_2(\F_p)$. Then $V=\Mat_2(\F_p)$.	
\end{lemma}

\begin{proof}
Since $G$ has more than $p$ elements, there must be a non-scalar matrix in $G$. So let $B\in G$ be a non-scalar matrix. Then since scalar matrices are the only matrices that commute with all elements on $\SL_2(\F_p)$, there must be $A\in \SL_2(\F_p)$ such that $A B A^{-1} \neq B$. Then $\tr (A B A^{-1} - B) = \tr (A B A^{-1}) - \tr (B) = \tr(B) - \tr (B) = 0$ and $V^0 := \{ B \in V | \tr (B) = 0 \} \neq \{ 0 \}$. Since the action by conjugation of $\SL_2 (\F_p)$ on $\{ B \in \Mat_2 (\F_p) | \tr (B) = 0 \} $ leaves only the trivial subvector spaces invariant and $V^0$ is not just the zero vector, we find that $V^0 = \{ B \in \Mat_2(\F_p) | \tr (B) = 0 \}$ which has dimension $3$. Now since for $p > 2$ the identity matrix is an element of $V$, but not of $V^0$, we have $V \supsetneq V^0$.  But since $V^0$ is an $\F_p$-vector space of dimension $3$ (hence has order $p^3$) and $V$ is strictly larger than $V^0$ (hence has to have order strictly larger than $p^3$), we get $V=\Mat_2(\F_p)$.
\end{proof}

\section{The actual descent}

\begin{lemma}
\label{Lemma6.2}
Let $G := \Gal(F(N)/F(N/p))$. Suppose $E$ and $p$ satisfy (P1) and (P2). We assume $p^2\mid N$. Then:
\begin{itemize}
\item[(i)] The subgroup of $H := \Gal(K(N)/K)$ generated by the conjugates of $G$ equals $\Gal(K(N)/K(N/p))$.
\item[(ii)] If $\alpha\in K(N)$ with $\sigma(\alpha)\in F(N/p)$ for all $\sigma\in\Gal(K(N)/K)$, then $\alpha\in K(N/p).$
\end{itemize}
\end{lemma}

\begin{proof}
(i) We will follow closely the proof of Habegger's Lemma 6.2 but we will not use the concept of non-split Cartan subgroups as in Habegger's proof. Instead we will use the lemmas above to show that the group $H$ is large enough.

We have
\begin{align}
\label{6.3}
H  \subset \Gal(K(N)/K(N/p))
\end{align}
and we will show the equality.\\

We now want to look at the Galois representations and choose a basis for each $E[N]$ that is compatible with the diagram below. Let

$$\tilde{\rho} : \Gal(K(p)/K) \to \GL_2 (\F_p) \text{ and } \rho: \Gal(K(N)/K) \to \GL_2 (\Z/p^n\Z)$$ and put them into the following commutative diagram:

\begin{align}\label{Bild6.4}
\begin{xy}\xymatrix{
&&\Gal(K(N)/K) \ar[d] \ar[r]^\rho & \GL_2 (\Z/p^n\Z) \ar[d] \\
&&\Gal(K(N/p)/K) \ar[d] \ar[r]^{\rho|_{K(N/p)}} & \GL_2 (\Z/p^{n-1}\Z)\ar[d]\\
&&\Gal(K(p)/K) \ar[r]^{\tilde{\rho}} & \GL_2 (\F_p)
} \end{xy}
\end{align}

The right vertical arrows are the natural surjections and the left vertical arrows are induced by the restrictions. By Lemma \ref{containsSL2} we know that $\tilde{\rho}(\Gal(K(p)/K))$ contains $\SL_2(\F_p)$. We will now construct a homomorphism $\Le$ from $\Gal(K(N)/K(N/p))$ to $\Mat_2 (\F_p)$ which will firstly show by its injectivity that $|H| \leq p^4$ and secondly, through Lemma \ref{isMat2}, show equality.\\

If $\sigma\in\Gal(K(N)/K(N/p))$ then $\rho (\sigma)$ is represented by $1+p^{n-1} \Le^\prime (\sigma)$ with $\Le^\prime (\sigma)\in\Mat_2 (\Z)$. Moreover, $\Le^\prime (\sigma)$ is well-defined modulo $p\Mat_2 (\Z)$. We obtain by reduction mod $p$ a "logarithm" $\Le : \Gal (K(N)/K(N/p)) \to \Mat_2 (\F_p)$. The name comes from the following property: Let $\sigma_1, \sigma_2 \in \Gal (K(N)/K(N/p))$, then

\begin{center}
$\rho (\sigma_1 \sigma_2) = (1+p^{n-1} \Le (\sigma_1))(1+p^{n-1} \Le (\sigma_2)) \equiv 1+p^{n-1}(\Le (\sigma_1)+\Le (\sigma_2)) \mod p^n \Mat_2(\Z)
$\end{center} where we let $\Le$ be the reduction of $\Le^\prime$ modulo $p\Mat_2 (\Z)$. So $\Le(\sigma_1\sigma_2)=\Le(\sigma_1) + \Le(\sigma_2)$, hence $\Le$ is a group homomorphism. We want to show that $\Le$ is injective. Let $\sigma \in \Gal(K(N)/K(N/p))$ be such that $\Le (\sigma) = \id$ in $\Mat_2 \F_p$. This means that $\rho(\sigma) = 1$ in $\GL_2 (\Z/p^n \Z)$. We look at the diagram $\eqref{Bild6.4}$ and see that this means that $\sigma$ fixes $K(p^n)$. Since it is an element of $\Gal(K(N)/K(N/p)$, it also fixes $K(N/p)$, hence fixes $K(N)$. So $\sigma \in \Gal(K(N)/K(N/p))$ is the identity and $\Le$ is injective.

Hence we get
\begin{align}
[K(N):K(N/p)] \leq |\Mat_2(\F_p)| = p^4. \label{6.7}
\end{align}

Remark that since the Galois extension $K(N/p)/K$ is normal, the Galois group $\Gal(K(N)/K(N/p))$ is a normal subgroup of $\Gal(K(N)/K)$, hence if $\sigma\in \Gal(K(N)/K)$ and $\psi \in G$ we have $\sigma\psi\sigma^{-1}\in\Gal(K(N)/K(N/p))$. Then

\begin{align*}
\rho (\sigma\psi\sigma^{-1}) \equiv 1+p^{n-1}\tilde{\rho}(\sigma) \Le^\prime(\psi)\tilde{\rho}(\sigma)^{-1} \mod p^n\Mat_2(\Z).
\end{align*}
Hence 
\begin{align}
\Le(\sigma\psi\sigma^{-1}) = \tilde{\rho}(\sigma) \Le(\psi)\tilde{\rho}(\sigma)^{-1}. \label{equation6.8}
\end{align}
By Lemma \ref{Lemma3.4} (iv), $G$ has order $p^2$, so $|\Le (G)| = p^2$ by injectivity of $\Le$.\\

We recall that by Lemma \ref{containsSL2} the image of $\tilde{\rho}$ contains $\SL_2 (\F_p)$. So the definition of $H$ and equation \eqref{equation6.8} imply that conjugating a matrix in $\Le(G)$ by an element of $\SL_2(\F_p)$ stays within $\Le(H)$. We want to apply Lemma \ref{isMat2} to $\Le(G)$ to deduce $\Le(H) = \Mat_2(\F_p)$. Then $|\Le(H)| = |H| = p^4$ and by \eqref{6.7}, $H$ has to be equal to $\Gal(K(N)/K(N/p))$. For applying Lemma \ref{isMat2} we have to prove that $\Le (G)$ contains a non-zero scalar matrix:\\

By Lemma \ref{Lemma3.4} (v) we know that the image of the Galois representation \newline $\Gal(F(p^n)/F) \to \Aut E[p^n]$ contains multiplication by $M^{[F:\Q_q]}$ for any $M$ coprime to $p$. We now want to construct an element in $\Gal(F(p^n)/F(p^{n-1}))$ whose image is scalar multiplication. Let $M$ be a generator of $(\Z/p^n\Z)^*$, then the multiplication by $M$ will have order $(p-1)p^{n-1}$ in $\Aut E[p^n]$.\\

Now by Lemma \ref{Lemma6.4} we know that there exists $\sigma_F^\prime\in\Gal(F(p^n)/F)$ such that its image is the multiplication by $M^{[F:\Q_q]}$. We want to show that $\sigma_F^{\prime (q-1)p^{n-2}} $ is an element of $\Gal(F(p^n)/F(p^{n-1}))$ and that it is not trivial. We start with the non-triviality. Remark that $\Gal(F(p^n)/F)/\Gal(F(p^n)/F(p^{n-1})) \cong \Gal(F(p^{n-1})/F)$. Since $\Gal(F(p^n)/F)$ is isomorphic to $A\times (\Z/p^{n-1}\Z)^2$ where $A$ is a subgroup of $\Z/(q-1)\Z$ (see the proof of Lemma \ref{Lemma3.5}), its exponent will be $ap^{n-1}$ where $a\mid (q-1)$. But $[F:\Q_q]$ is not a multiple of $ap^{n-1}$ since $[F:\Q_q]$ is coprime to $p$. Hence $\sigma_F^{\prime (q-1)p^{n-2}}$ is not trivial.\\

Now consider $\Gal(F(p^{n-1})/F)$. This group is isomorphic to $A\times (\Z/p^{n-2}\Z)^2$ where $A$ is a subgroup of $\Z/(q-1)\Z$ and its exponent will be $ap^{n-2}$ where $a\mid (q-1)$. On the other hand, $[F:\Q_q](q-1)p^{n-2}$ is now a multiple of $(q-1)p^{n-2}$, hence the restriction of $\sigma_F^{\prime (q-1)p^{n-2}}$ to $F(p^{n-1}))$ is trivial which means that it has to be in $\Gal(F(p^n)/F(p^{n-1}))$.\\

By Lemma \ref{Lemma3.4} (iv), we have $\Gal(F(p^n)/F(p^{n-1}))\cong \Gal(F(N)/F(N/p))$, hence we can find $\sigma_F \in \Gal(F(N)/F(N/p))$ that gets mapped to $\sigma_F^\prime$ under that isomorphism. We can apply $\Le$ to find that $\Le(\Gal(F(N)/F(N/p)))$ contains an element that acts as scalar multiplication on the $p^n$-torsion points, hence has to be a scalar matrix.\\

(ii) Now we proceed with the second part. Let $\alpha\in K(N)$ with $\sigma(\alpha)\in F(N/p)$ for all $\sigma\in\Gal(K(N)/K)$. Since we can invert elements of Galois groups, it makes sense to consider $\sigma^{-1}$ whenever $\sigma\in\Gal(K(N)/K)$ and with the first part of the Lemma we get that the group generated by $\sigma\psi\sigma^{-1}$ equals $\Gal(K(N)/K(N/p))$. Since $\alpha$ is fixed by such a $\sigma\psi\sigma^{-1}$, it has to be in $K(N/p)$ which is what we wanted to show.
\end{proof}

The technique of the descent used in the following theorem has been developed by Amoroso and Zannier in Section 4 of \cite{MR2651944}.

\begin{theorem}
\label{Proposition6.1}
Let $E$ be an elliptic curve over $\Q$ without complex multiplication. Let $L$ be a Galois extension of $\Q$. Suppose there exists $d\in\N$ such that $L$ has uniformly bounded local degrees above all but finitely many primes where $d$ is the said uniform bound. Then there is a prime number $p$ satisfying \eqref{P1}, \eqref{P2}, \eqref{P3} and \eqref{P4}. If $\alpha\in L(E_{\text{tor}})^* \backslash \mu_\infty$, then
\begin{align}
\label{Equation6.11}
h(\alpha) \geq \frac{(\log p)^4}{p^{5p^4}}.
\end{align}
\end{theorem}

\begin{proof}
Again, we follow here the analogous proof of Proposition 6.1 of \cite{MR3090783} closely. Since $E$ does not have complex multiplication, its $j$-invariant is neither $0$, nor $1728$. So the reduction of $E$ at $p$ is an elliptic curve with $j$-invariant neither $0$, nor $1728$ for all but finitely many primes $p$. By a Theorem of Serre \cite{MR0387283}, all but finitely many of these $p$ satisfy (P2). Furthermore, by \cite{MR903384}, there are infinitely many supersingular primes for an elliptic curve over $\Q$. We may thus fix a prime $p$ satisfying (P1), (P2), (P3) and (P4) and set $q=p^2$.\\

Recall the following facts thet we fixed in the beginning of the chapter:
Let $\alpha \in L(E_{\text{tor}})^* \backslash \mu_\infty $. Then $\alpha \in K(N)$ for some $N=p^n M$ with $M\in\N$ coprime to $p$, $n$ a nonnegative integer and $K \subset L$ a number field that is Galois over $\Q$. Let $\Q_q$ be the unique quadratic unramified extension of $\Q_p$. Then we fix a finite Galois extension $F/\Q_q$ with $\Q_q \subset F \subset \overline{\Q_p}$ such that the $v$-adic completion of $K$ is contained in $F$ (where $v$ extends $p$) and $[F:\Q_p]$ is uniformly bounded by $d$. Let furthermore $\mathcal{E} = (q-1)[F:\Q_q] \exp(\Gal(L/\Q))$.\\

We take $\sigma_F\in \Gal(F(N)/F)$ as in Lemma \ref{Lemma6.4}. If we are in the case of $p^2 \nmid N$, we can artificially choose an element in $\Gal(F(p^2N)/F)$ and restrict it to $F(N)$. Then we define
\begin{align}
\label{Equation6.12}
\gamma = \frac{\sigma_F (\alpha)}{\alpha^{4^{[F:\Q_q]}}} \in K(N).
\end{align}
By the properties of the height we get
\begin{align}
\label{Equation6.13}
h(\gamma) \leq h(\sigma_F (\alpha)) + h(\alpha^{4^{[F:\Q_q]}}) = (4^{[F:\Q_q]} +1) h(\alpha).
\end{align}

Let us start with the case of $n\geq 2$, hence $p^2 \mid N$. Since $\gamma\in K(N) \subset F(N)$, hence $\sigma (\gamma) \in K(N) \subset F(N)$ for all $\sigma\in\Gal(K(N)/K)$, there is a least integer $n^\prime \leq n$ such that $\sigma(\gamma)\in F (p^{n^\prime}M)$ for all $\sigma\in\Gal(K(N)/K)$. Lemma \ref{Lemma6.2} implies that then also $\gamma\in K(p^{n^\prime}M)$.\\

By minimality of $n^\prime$ there is a $\sigma\in\Gal(K(N)/K)$ such that $\sigma(\gamma)\notin F (p^{n^\prime -1}M)$.  We will split this up into two cases: First $n^\prime \geq 2$ and second $n^\prime \leq 1$. We start with $n^\prime \geq 2$. We apply $\sigma$ to \eqref{Equation6.12} and obtain
\begin{align}
\label{Equation6.14}
\sigma(\gamma) = \frac{\sigma(\sigma_F (\alpha))}{\sigma(\alpha^{4^{[F:\Q_q]}})} = \frac{\sigma_F(\sigma (\alpha))}{\sigma(\alpha)^{4^{[F:\Q_q]}}}
\end{align}
since $\sigma_F$ lies in the center of $\Gal(K(N)/K)$ by Lemma \ref{Lemma6.4}.

Next we want to apply Lemma \ref{Lemma5.3} to $\sigma(\gamma)$, so we must verify that $\sigma(\gamma)^q\notin F(p^{n^\prime -1}M)$. We will show this by contradiction, so assume $\sigma(\gamma)^q\in F(p^{n^\prime -1}M)$.

Since $\sigma(\gamma)\notin F (p^{n^\prime -1}M)$, there is $\psi \in \Gal(F (N)/F(p^{n^\prime-1}M))$ such that $\psi(\sigma(\gamma)) \neq \sigma(\gamma)$. Furthermore, $\psi(\sigma(\gamma)^q) = \sigma(\gamma)^q$ by our assumption and so
\begin{align}
\label{Equation6.15}
\psi(\sigma(\gamma)) = \xi \sigma(\gamma) \text{ for some } \xi^q = 1 \text{ while } \xi \neq 1.
\end{align}

We apply $\psi$ to equation \eqref{Equation6.14} and obtain
\begin{align*}
\psi(\sigma(\gamma)) & = \frac{\psi(\sigma_F(\sigma (\alpha)))}{\psi(\sigma(\alpha)^{4^{[F:\Q_q]}})}\\
& = \frac{\sigma_F(\psi (\sigma(\alpha)))}{\psi(\sigma(\alpha))^{4^{[F:\Q_q]}}} 
\end{align*}
since $\psi$ commutes with $\sigma_F$ (by Lemma \ref{Lemma6.4}). We define $\eta = \frac{\psi(\sigma(\alpha))}{\sigma(\alpha)} \neq 0$ and get, by \eqref{Equation6.14} and \eqref{Equation6.15},
\begin{align}
\xi & = \frac{\psi(\sigma(\gamma))}{\sigma(\gamma)} \nonumber\\
& = \frac{\psi\left( \frac{\sigma_F(\sigma(\alpha))}{\sigma(\alpha)^{4^{[F:\Q_q]}}} \right)}{\frac{\sigma_F(\sigma(\alpha))}{\sigma(\alpha)^{4^{[F:\Q_q]}}}} \nonumber\\
& = \frac{\psi(\sigma_F(\sigma(\alpha)))\sigma(\alpha)^{4^{[F:\Q_q]}}}{\psi(\sigma(\alpha))^{4^{[F:\Q_q]}} \sigma_F (\sigma(\alpha))} \nonumber\\
& = \frac{\sigma(\alpha)^{4^{[F:\Q_q]}}} {\psi(\sigma(\alpha))^{4^{[F:\Q_q]}}} \frac{\psi(\sigma_F(\sigma(\alpha)))}{\sigma_F (\sigma(\alpha))} \nonumber\\ 
& = \frac{\sigma(\alpha)^{4^{[F:\Q_q]}}}{\psi(\sigma(\alpha))^{4^{[F:\Q_q]}}} \frac{\sigma_F(\psi(\sigma(\alpha)))}{\sigma_F (\sigma(\alpha))} \nonumber\\
& = \eta^{-4^{[F:\Q_q]}} \sigma_F(\eta)\nonumber\\
& = \frac{\sigma_F(\eta)}{\eta^{4^{[F:\Q_q]}}} \label{XiEqualsXiPrime}.
\end{align}
Since $\xi$ is a root of unity, we have $$4^{[F:\Q_q]} h(\eta) = h(\eta^{4^{[F:\Q_q]}}) = h(\xi \eta^{4^{[F:\Q_q]}}) = h(\sigma_F (\eta)) = h(\eta),$$ so $h(\eta) = 0$ and by Kronecker's Theorem, $\eta$ is a root of unity.

We now fix $\widetilde{M}\in\N$ coprime to $p$ such that $\eta^{\widetilde{M}}\in\mu_{p^\infty}$. Lemma \ref{Lemma3.5} now implies that $\eta^{\widetilde{M}}$ is already in $\mu_{p^n}$ and by Lemma \ref{Lemma6.4} we have $\sigma_F (\eta^{\widetilde{M}}) =(\eta^{\widetilde{M}})^{4^{[F:\Q_q]}}$. And we get
\begin{align}
\sigma_F (\eta) = \xi^\prime \eta^{4^{[F:\Q_q]}}
\end{align}
for some $\xi^\prime$ such that $\xi^{\prime{\widetilde{M}}} = 1$. Using equation \eqref{Equation6.14} and \eqref{XiEqualsXiPrime} we get that $\xi = \xi^\prime$ and hence $\xi^q = \xi^{\widetilde{M}} = 1$. But since $\widetilde{M}$ and $q$ are coprime we must have $\xi = 1$ which is a contradiction. So our assumption on $\sigma(\gamma)^q$ is false and we get $\sigma(\gamma)^q \notin F (p^{n^\prime-1}M)$. Hence we can apply Lemma \ref{Lemma5.3} and get the following lower bound for the height of $\sigma(\gamma)$
\begin{align}
\frac{(\log p)^4}{4\cdot 10^6 p^{32}} \leq  h(\sigma (\gamma)) = h(\gamma) \leq (4^{[F:\Q_q]}+1) h(\alpha) \label{ramifiedHeight}.
\end{align}
This was the case of $n^\prime \geq 2$. Let us now assume that $n^\prime \leq 1$, which allows us to descent to the tamely ramified case. Recall that $\gamma = \frac{\sigma_F (\alpha)}{\alpha^{4^{[F:\Q_q]}}}$. We do a descent as in the totally ramified case: There is a least integer $n^\prime \leq 1$ such that $\sigma(\gamma)\in F (p^{n^\prime}M)$ for all $\sigma\in\Gal(K(N)/K)$. Lemma \ref{Lemma6.2} implies that then also $\gamma\in K(p^{n^\prime}M)$. We will treat this case together with the case $n \leq 1$ where we do not need a descent at all.\\

Since $h(\sigma(\gamma)) = h(\gamma)$ we can in both cases compute the height of $\gamma$ where $\gamma$ will be an element of $K(p^{n^\prime}M)$ where $n^\prime \leq 1$. We want to apply Lemma \ref{Lemma6.3}, so we will prove that $\gamma \neq 0$ and not a root of unity. If so, we would have $4^{[F:\Q_q]} h(\alpha) = h(\alpha^{4^{[F:\Q_q]}}) = h(\gamma \alpha^{4^{[F:\Q_q]}}) = h(\sigma_F (\alpha)) = h(\alpha)$ by the properties of the height and hence $h(\alpha) = 0$. By Kronecker's Theorem this either means $\alpha =0$ or $\alpha\in\mu_\infty$. But this is a contradiction to our assumption on $\alpha$. Hence Lemma \ref{Lemma6.3} gives 
\begin{align*}
h(\gamma) \geq  \left(\frac{\log p}{\mathcal{E}(1+q^\mathcal{E})(1+  \frac1{2^{10}})}\right)^4.
\end{align*}
Moreover, we can use inequality \eqref{Equation6.13} and $\mathcal{E} \leq \frac{p^4}{2}$ to get
\begin{align*}
h(\alpha) & \geq \frac1{4^{[F:\Q_q]}+1}  \left(\frac{\log p}{\mathcal{E}(1+q^\mathcal{E})(1+  \frac1{2^{10}})}\right)^4\\
& \geq \frac{(\log p)^4}{(4^{[F:\Q_q]}\mathcal{E}q^\mathcal{E}2)^4}\\
& \geq \frac{(\log p)^4}{(4^{\frac{p^4}{2}} \frac{p^4}{2}p^{p^4}2)^4}\\
& \geq \frac{(\log p)^4}{(4^{\frac{p^4}{2}} p^{p^4+4})^4} \\
& \geq \frac{(\log p)^4}{4^{2p^4} p^{4p^4+16}} \\
& \geq \frac{(\log p)^4}{p^{\frac{\log 4}{\log p}2p^4 + 4p^4+16}} \\
& \geq \frac{(\log p)^4}{p^{5p^4}}.
\end{align*}
Now we have to put the tamely (inequality above) and totally ramified  (inequality \eqref{ramifiedHeight}) case into one bound.\\

We get $h(\alpha) \geq \max\left(\frac{(\log p)^4}{(4^{[F:\Q_q]}+1)\cdot 2\cdot 10^3 p^{32}}, \frac{(\log p)^4}{p^{5p^4}}\right) = \frac{(\log p)^4}{p^{5p^4}}$. 
\end{proof}

\bibliography{Bibliographie}

\def\cprime{$'$}
\begin{thebibliography}{{Fre}17}

\bibitem[ADZ14]{MR3182009}
Francesco Amoroso, Sinnou David, and Umberto Zannier.
\newblock On fields with {P}roperty ({B}).
\newblock {\em Proc. Amer. Math. Soc.}, 142(6):1893--1910, 2014.

\bibitem[AZ10]{MR2651944}
F.~Amoroso and U.~Zannier.
\newblock A uniform relative {D}obrowolski's lower bound over abelian
  extensions.
\newblock {\em Bull. Lond. Math. Soc.}, 42(3):489--498, 2010.

\bibitem[Che13]{MR3009657}
S.~Checcoli.
\newblock Fields of algebraic numbers with bounded local degrees and their
  properties.
\newblock {\em Trans. Amer. Math. Soc.}, 365(4):2223--2240, 2013.

\bibitem[CZ11]{MR2755687}
Sara Checcoli and Umberto Zannier.
\newblock On fields of algebraic numbers with bounded local degrees.
\newblock {\em C. R. Math. Acad. Sci. Paris}, 349(1-2):11--14, 2011.

\bibitem[Elk87]{MR903384}
N.~D. Elkies.
\newblock The existence of infinitely many supersingular primes for every
  elliptic curve over {${\mathbf Q}$}.
\newblock {\em Invent. Math.}, 89(3):561--567, 1987.

\bibitem[{Fre}17]{2017arXiv171204214F}
L.~{Frey}.
\newblock {Explicit Small Heights in Infinite Non-Abelian Extensions}.
\newblock {\em ArXiv e-prints}, December 2017.

\bibitem[Gou89]{MR1508819}
Edouard Goursat.
\newblock Sur les substitutions orthogonales et les divisions r\'eguli\`eres de
  l'espace.
\newblock {\em Ann. Sci. \'Ecole Norm. Sup. (3)}, 6:9--102, 1889.

\bibitem[Hab13]{MR3090783}
P.~Habegger.
\newblock Small height and infinite nonabelian extensions.
\newblock {\em Duke Math. J.}, 162(11):2027--2076, 2013.

\bibitem[Lan02]{MR1878556}
Serge Lang.
\newblock {\em Algebra}, volume 211 of {\em Graduate Texts in Mathematics}.
\newblock Springer-Verlag, New York, third edition, 2002.

\bibitem[Neu99]{MR1697859}
J.~Neukirch.
\newblock {\em Algebraic number theory}, volume 322 of {\em Grundlehren der
  Mathematischen Wissenschaften [Fundamental Principles of Mathematical
  Sciences]}.
\newblock Springer-Verlag, Berlin, 1999.
\newblock Translated from the 1992 German original and with a note by Norbert
  Schappacher, With a foreword by G. Harder.

\bibitem[Ser72]{MR0387283}
J.-P. Serre.
\newblock Propri\'et\'es galoisiennes des points d'ordre fini des courbes
  elliptiques.
\newblock {\em Invent. Math.}, 15(4):259--331, 1972.

\bibitem[Ser79]{MR554237}
Jean-Pierre Serre.
\newblock {\em Local fields}, volume~67 of {\em Graduate Texts in Mathematics}.
\newblock Springer-Verlag, New York-Berlin, 1979.
\newblock Translated from the French by Marvin Jay Greenberg.

\bibitem[Sil09]{MR2514094}
J.~H. Silverman.
\newblock {\em The arithmetic of elliptic curves}, volume 106 of {\em Graduate
  Texts in Mathematics}.
\newblock Springer, Dordrecht, second edition, 2009.

\end{thebibliography}
\bibliographystyle{alpha}
\parindent0pt
\end{document}